\documentclass[10pt,a4wide]{article} 
\usepackage{bbm,amsmath,amsthm, amssymb}
\usepackage[active]{srcltx}
\usepackage{color}
\usepackage{macros}
\usepackage{enumerate}
\definecolor{rltred}{rgb}{0.75,0,0}
\definecolor{rltgreen}{rgb}{0,0.5,0}
\definecolor{rltblue}{rgb}{0,0,0.75}
\usepackage{dirtytalk}
\usepackage[%
   colorlinks=true,
   urlcolor=rltblue,       
   filecolor=rltgreen,     
   citecolor=rltgreen,     
   linkcolor=rltred,       
   bookmarks=true,
   unicode,
   plainpages=false,
   ]{hyperref}
\begin{document}
\title{\bf On the Reynolds time-averaged equations and the long-time behavior of
  Leray-Hopf weak solutions, with applications to ensemble averages}
\author{Luigi C. Berselli
  \\
  Universit\`a di Pisa
  \\
  Dipartimento di Matematica
  \\
  Via Buonarroti~1/c, I-56127 Pisa, ITALY.
  \\
  email: luigi.carlo.berselli@unipi.it
  \\
  \\
  \and Roger Lewandowski
  \\
University of Rennes \& INRIA
  \\
CNRS, IRMAR, UMR 6625 \& Fluminance team, 
  \\
  Campus Beaulieu, F-35042 Rennes
  \\
  email: Roger.Lewandowski@univ-rennes1.fr }
\date{}
\maketitle
\begin{abstract}
  We consider the three dimensional incompressible Navier-Stokes equations with non
  stationary source terms $\fv$ chosen in a suitable space. We prove the existence of
  Leray-Hopf weak solutions and that it is possible to characterize (up to sub-sequences)
  their long-time averages, which satisfy the Reynolds averaged equations, involving a
  Reynolds stress. Moreover, we show that the turbulent dissipation is bounded by the sum
  of the Reynolds stress work and of the external turbulent fluxes, without any additional
  assumption, than that of dealing with Leray-Hopf weak solutions.

  Finally, in the last section we consider ensemble averages of solutions, associated to a
  set of different forces and we prove that the fluctuations continue to have a
  dissipative effect on the mean flow.

\medskip


\medskip


\end{abstract}
\section{Introduction}
Let us consider the 3D homogeneous incompressible Navier-Stokes equations (NSE in the
sequel),
\begin{equation}
  \label{eq:NSE}
  \left\{  
    \begin{aligned}
      \bv_{t}+ (\bv\cdot\nabla)\, \bv-\nu\Delta \bv+\nabla
      p&=\bff\qquad &\text{ in }]0,+\infty[ \times\Omega,
      \\
      \nabla\cdot \bv&=0\qquad &\text{ in }]0,+\infty[\times\Omega,
      \\
      \bv&=\mathbf{0}\qquad &\text{ on }]0,+\infty[\times\Gamma,
      \\
      \bv(0, \x)&=\bv_{0}(\x) \qquad &\text{ in }\Omega,
  \end{aligned}
\right.
\end{equation} 
where $\Om \subset \R^3$ is a bounded Lipschitz domain, $\Ga = \p \Om$ its boundary, $\vv=
\vv(t, \x)$ denotes the fluid velocity, $p = p(t, \x)$ the pressure, $\fv = \fv (t, \x)$
the external source term, and $(t, \x) \in ]0,+\infty[\times\Omega$. The main aim of this
paper is to study the long-time averages of weak solutions to the NSE~\eqref{eq:NSE},
namely 
\begin{equation}
  \label{eq:Moyenne} \vm (\x) := \lim_{t \to \infty} M_t (\vv), \quad\text{where}\quad M_t(\vv) :=
  \frac{1}{t} \int_0^t \vv (s, \x)\, ds,
\end{equation}
when the source term $\fv$ is time dependent, and to link it to the Reynolds averaged
equations (see~\eqref{eq:Reynolds_eqs} below). We also will consider the problem of
ensemble averages, which is closely related to long-time averages.

Long-time average for ``tumultuous'' flows, today turbulent flows, seem to have been
considered first by G.~Stokes~\cite{2GS51} and then by O.~Reynolds~\cite{Rey1895}. The
idea is that \textit{steady-state} turbulent flows are oscillating around a stationary
flow, which can be expressed through long-time averages. L.~Prandtl used them to introduce
the legendary ``Prandtl mixing length'' (see~\cite{LP25} and in~\cite[Ch.~3,
Sec.~4]{LP52}) to model a turbulent boundary layer over a plate.  However, although
long-time averaging plays a central role in turbulence modeling (since it is natural as
well as ``ensemble averaging'') it is not clear whether the limit~\eqref{eq:Moyenne} is
well defined, or not.

The mathematical problem of properly defining long-time averages and to investigate the
connection with the Reynolds equations was already studied before, but when the source
term $\bf$ does not depend on time, namely $\fv = \fv (\x)$.

So far as we know, the first who considered this issue is
C.~Foias~\cite[Sec.~8]{Foi197273}, when $\fv \in H$, where
\begin{equation}
  \label{eq:espace_H} 
  H := \{ \uv \in L^2(\Om)^3 \ \hbox{s.t.} \, \ \g \cdot \uv = 0,\ \uv\cdot\mathbf{n}=0 \text{ on }\Gamma\}.
\end{equation} 
Foias analysis is based on the notion of ``statistical solution'' introduced
in~\cite[Sec.~3]{Foi197273} and he was able to prove that for any given Leray-Hopf
solution $\vv = \vv(t,\x)$ to~\eqref{eq:NSE} then:
 
\smallskip 

\noindent - There exists $\vm \in H$, such that, up to a sub-sequence, $M_{t} (\vv) \to
\vm$ in $H$ as $t \to \infty$;

\noindent - There is a \textit{stationary statistical solution} $\mu$ to the NSE, which is a probability
measure on $H$, such that it holds $\vm = \int_H \wv\, d\mu (\wv)$;
 
\noindent - There exists a random variable on $(H, \mu)$ called $\bv'$ (in the notation of
Foias called $\delta\bv$) such that
\begin{equation*}
  \mean{\bv'}=\int_{H}\bv'\,d\mu=0\qquad\text{and}\qquad\mean{\|\nabla\bv'\|^{2}}=\int_{H}\|\nabla
  \bv'\|^{2}\,d\mu<\infty, 
\end{equation*}
such that $(\vm, \vv')$ is a solution to the Reynolds Equations given
by~\eqref{eq:Reynolds_eqs} below. If $\mu$ is a statistical solution, $\vv'$ is the random
variable expressed by its probability law, 
\begin{equation}
  \label{eq:fluctuation_foias_1}
\text{Prob} ( \bv'\in F) = \mu (\vm+ F), 
\end{equation}
for any Borel
set $F\in H$.
 
 \noindent - According to our notations, the Reynolds stress $\reyn$ given by
 \begin{equation*}
   \g \cdot \reyn = \int_H \g \cdot \big[(\vv-\vm) \otimes (\vv-\vm)\big]\,d\mu (\vv),
\end{equation*}
is \textit{dissipative on the mean flow,} which is one of the main challenges of such
analysis, because of the Boussinesq assumption (see
in~\cite[Ch.~4,\,Sec.~4.4.3.1]{CL2014}). In fact this is much more better, since Foias
in~\cite[Sec.~8-2-a, Prop.~1, p.~99]{Foi197273} was able to prove that the \textit{turbulent
dissipation} $\E$ is bounded by the work of the Reynolds stress on the mean flow,
\begin{equation}
  \label{eq:fluctuation_foias} 
  \E := \nu \overline{ \| \g \vv ' \|^2} \le (\g \cdot\reyn, \vm),
\end{equation}
where $(\cdot, \cdot)$ and $\|\,.\,\|$ denote the standard $L^2(\Omega)$ scalar product
and norm, respectively. (Sometimes norm is normalized by $|\Omega|$, the measure of the
domain, but this is inessential). This analysis is partially reported in~\cite[Ch.~3,
Sec.~3]{FMRT2001a}, where the limit in~\eqref{eq:Moyenne} is replaced by the abstract
Banach limit, but without any link to the Reynolds equations and dissipation inequality
such as~\eqref{eq:fluctuation_foias}. 

The original Foias analysis is very deep and essential to the field. However, it is worth
noting that in this approach:

\smallskip 

\noindent i) The natural time filter used to determine the Reynolds stress (initially
suggested by Prandtl, when the limit makes sense) 
\begin{equation*}
  \reyn = \reyn (\x) := \lim_{t \to
    \infty} \frac{1}{t} \int_0^t [(\vv - \vm) \otimes (\vv - \vm)](s, \x)\, ds,
\end{equation*}
is replaced by an abstract probability measure that it is not possible to calculate in
practical simulations, although the ergodic assumption -which remains to be proved- would
mean that they coincide (see for instance Frisch~\cite{UF95});

\smallskip 

\noindent ii) The fluctuation given by~\eqref{eq:fluctuation_foias_1}, when the force is
time independent, is a time independent random variable, which may be questionable from
the physical point of view. In fact, when $\vm$ is given by~\eqref{eq:Moyenne}, one cannot
conclude that this $\vv'$ yields the Reynolds decomposition
\begin{equation*}
\vv(t, \x) = \vm (\x) + \vv'(t, \x),
\end{equation*}
in which the fluctuation is time dependent for a realistic non stationary flow.

\smallskip 

\noindent iii) Concerning items i) and ii) above, we can suggest that probably there are
still work to be done in the interesting field of statistical solutions, if one wants to
use this mathematical tool to face the tough closure problem in the Reynolds equations,
and apply to realistic problem. Moreover, the case of a time dependent source term also
remains to be considered in connection with the scope of statistical solutions (see the
warning in~\cite[Part~I, Sec.~5, p.~313]{Foi197273}).
 
\smallskip

Still in the case of a stationary source term, the long-time averaged problem has been
also considered more recently in~\cite{Lew2015} and at the time the author was not aware
yet of the connection between time-averaging and statistical solutions in Foias
work. In~\cite{Lew2015}, he studied the equation satisfies by $M_t(\vv)$ (see
equations~\eqref{eq:Reynolds_weak_form} below) for a Leray-Hopf weak solution of the NSE,
and he took the limit in this equation when the domain if of class $C^{9/4, 1}$ and
$\bff=\bff(\bx)\in L^{5/4}(\Om)^3 \cap V'$, where
\begin{equation}
  \label{eq:espace_V} 
  V:= \{ \uv \in H^1_0(\Om)^3 \,\, \hbox{s.t.} \, \ \g \cdot \uv = 0 \},
\end{equation}
and $V'$ denotes the dual space of $V$, with duality pairing $<.\,,\,.>$.  The analysis is
based on the energy inequality, which yields a uniform estimate in time of the $L^2$ norm
of $\vv$, and on a $L^{p}$-regularity result by Amrouche and Girault~\cite{2AG91} about
the steady Stokes problems, valid in $C^{k, \alpha}$ domains. It is shown
in~\cite{Lew2015} that there exists $\reyn \in L^{5/3} (\Om)^9$ and $\plm \in W^{1, 5/4}
(\Om)/\R$ such that, up to a sub-sequence, when $t \to \infty$, $M_t (\vv)$ converges to a
field $\vm\in {\bf W}^{2, 5/4}(\Om)^3$, which satisfies in the sense of the distributions
the closed Reynolds equations:
 \begin{equation}
  \label{eq:Reynolds_Stat_pb_2}
  \left\{
    \begin{aligned}
     (\mean{\bv}\cdot\nabla)\, \mean{\bv} -\nu\Delta \mean{\bv}+\nabla
      \mean{p}+\nabla\cdot\reyn&= \fv\qquad &\text{ in }\Omega,
      \\
      \nabla\cdot \mean{\bv}&=0\qquad &\text{ in }\Omega,
      \\
      \mean{\bv}&=\mathbf{0}\qquad &\text{ on }\Gamma.
    \end{aligned}  
  \right.
\end{equation}
Moreover, it is also shown that $\reyn$ is dissipative on the mean flow, namely
 \begin{equation}
  \label{eq:reyn_dissipative} 
  0 \le (\nabla\cdot \reyn, \mean{\bv}),  
\end{equation}
which is weaker than~\eqref{eq:fluctuation_foias}. 
 
The main part of the present study is in continuation of~\cite{Lew2015}, bringing
substantial improvements. The novelty is that we are considering time dependent source
terms $\fv = \fv(t, \x)$, which was never considered before for this problem, up to our
knowledge. Moreover, we do not need extra regularity assumption on the domain $\Om\subset
\R^{3}$ and on $\bff$. The first main result of this paper, Theorem~\ref{thm:main_theorem} below, is
close to that proved in~\cite{Lew2015}. Roughly speaking, we will show that $M_t (\vv)$
given by~\eqref{eq:Moyenne}, converges to some $\vm$ (up to sub-sequences) and that there
are $\plm$ and $\reyn$ such that~\eqref{eq:Reynolds_Stat_pb_2} holds, at least in ${\cal
  D}'(\Om)$, in which $\fv$ is replaced by $\overline {\fv}$.
 
One key point is the determination of a suitable class for the source term. Throughout the
paper, we will take $\fv:\ \R_+\to V'$, made of function for which there is a
constant $C>0$ such that 
\begin{equation*}
\forall \, t \in \R_+\qquad \int_t^{t+1} \| \fv (s) \|^{2}_{V'}\, ds \le C.
\end{equation*}
In this respect we observe our results improve the previous ones also in terms of
regularity of the force, not only because we consider a time-dependent one. The main
building block of our work is the derivation of a uniform estimate of the $L^2$ norm of
$\vv$, for $\fv$ as above (see~\eqref{eq:main_estimate_1} and its
Corollary~\eqref{eq:main_estimate_2}). This allows to prove the existence of a weak
solution on $[0,\infty)$ to the NSE and to pass to the limit in the equation satisfied by
$M_t(\vv)$, when $t \to \infty$.

Moreover, we are able to generalise Foias result~\eqref{eq:fluctuation_foias} by proving
that the turbulent dissipation $\E$ is bounded by the sum of the work of the Reynolds
stress on the mean flow and of the external turbulent fluxes, namely
\begin{equation}
\label{eq:fluctuation_LB_RL} \E = \nu \overline{ \| \g \vv ' \| } \le (\g \cdot
\reyn, \vm) + \overline {< \fv', \vv' > },
\end{equation}
which is one of the main features of our result.  Note that physically, it is expected
that~\eqref{eq:fluctuation_LB_RL} becomes an equality for strong solutions.  Furthermore,
we also prove that when $\fv$ is ``attracted'' (in some sense,
see~\eqref{eq:additional_ass_f} below) by a stationary force $\widetilde \fv = \widetilde
\fv (\x) $ as $t \to \infty$, then the turbulent fluxes $\overline {< \fv', \vv' > }$
in~\eqref{eq:fluctuation_LB_RL} vanish, so that~\eqref{eq:fluctuation_foias} is
restored. In particular, the question whether the stress tensor $\reyn$ is dissipative
remains an open problem for general unsteady $\fv$, as those for which we still have global
existence of weak solutions.

In the second part of the paper, we will also consider ensemble averages, often used in
practical experiments. This consists of considering $n\in\N$ realizations of the flow,
$\{\vv_1,\dots,\vv_n\}$ and to evaluate the arithmetic mean
\begin{equation*}
  \smean{\vv} := \frac{1}{n}\sum_{k=1}^n \vv_k.
\end{equation*}
Layton considers such ensemble averages in~\cite{Lay2014}, by introducing the
corresponding Reynolds stress, written as
\begin{equation*}
  R(\vv, \vv) =  \smean{ \vv   \otimes \vv} -\smean{\vv}\otimes \smean{\vv}.
\end{equation*}
He shows that for a fixed stationary source term $\fv$ and $n$ strong solutions of the
NSE, then $R(\vv, \vv)$ is dissipative on the ensemble average, in time average. More
specifically it holds 
\begin{equation*}
  \liminf_{t \to \infty} M_t [ (\g \cdot R(\vv, \vv), \smean{\vv}) ]  \ge 0,
\end{equation*}
and to prove this inequality he first performs the ensemble average, then it takes
the long-time average.

In order to remove the additional assumption used in~\cite{Lay2014,JL2016} of having strong
solutions, we will carry out an approach very close, but in a reversed order: We first take
the long-time average of the realizations for a sequence of time independent source terms
$\suite \fv k$. Then, we form the ensemble average
\begin{equation*}
    \mathbf{S}^n:= \frac{1}{n} \sum_{k=1}^n \vm_k,
\end{equation*}
where $\vm_k$ are weak solutions of the Reynolds equations.
Under suitable (but very light) regularity assumption about the $\fv_k$'s, we show the
convergence of $\suite S n$ to some $ \left < \vv \right > $ that satisfies the closed
Reynolds equations, and such that dissipativity still holds, that is
\begin{equation*}
  0 \le ( \g \cdot \smean{\reyn} ,  \smean{\vv} ),
\end{equation*} 
which holds for weak solutions (see the specific statement in
Theorem~\ref{thm:main_ensemble}).

\medskip

\noindent\textbf{Plan of the paper.}
The paper is organized as follows. We start by giving in Sec.~\ref{sec:main} the specific
technical statements of the results we prove. Then, in Sec.~\ref{sec:NSE-estimates} we give
some results about the functional spaces we are working with and we prove in detail the
main energy estimates~\eqref{eq:energy_1} and~\eqref{eq:dissipation} for source terms $\fv
\in L^2_{uloc}(\R_+, V')$, which yields an existence result of global weak solutions to
the NSE for such $\fv$.  The Sec.~\ref{sec:reynolds_rules} is devoted to the Reynolds
problem and the develop additional properties of the time average operator $M_t$. Finally,
we give the proofs of the two main results on time averages in Sec.~\ref{sec:proof-main1}
and Sec.~\ref{sec:ensemble}, respectively.

\section{Main results}
\label{sec:main}
\subsection{On the source term and an existence result} 
Since we aim to consider long-time averages for the NSE, we must consider solutions which
are global-in-time (defined for all positive times). Due to the well-known open problems
related to the NSE, this enforces us to restrict to weak solutions. By using a most
natural setting, we take the initial datum $\vv_0 \in H$, where $H$ is defined
by~\eqref{eq:espace_H}.


The classical Leray-Hopf results of existence (but without uniqueness) of a global weak
solution $\vv$ to the NSE holds when $\bff\in L^2(\R_+;V')$, and the velocity $\vv$
satisfies
\begin{equation*}
  \vv \in L^2(\R_+, V) \cap L^\infty(\R_+, H),
\end{equation*}
where $V$ is defined by~\eqref{eq:espace_V} 
and $V'$ denotes its topological dual. We will also denote by $<\, , \, >$ the duality
pairing\footnote{Generally speaking and when no risk of confusion occurs, we always denote
  by $<\, , \, > $ the duality pairing between any Banach space $X$ and its dual $X'$,
  without mentioning explicitly which spaces are involved.}  between $V'$ and $V$.

Source terms $\fv \in L^2(\R_+;V')$ verify $ \int_t^\infty \| {\bf f} (s) \|_{V'}^2\, ds
\to 0$ when $t \to+\infty$. Therefore, a turbulent motion cannot be maintained for large
$t$, which is not relevant for our purpose. The choice adopted in the previous studies on
Reynolds equations was that of a constant force, and we also observe that many estimates
could have been easily extended to a uniformly bounded $\bff\in L^{\infty}(\R_{+};V'$. On
the other hand, we consider a broader class for the source terms. According to the usual
folklore in mathematical analysis, we decided to consider the space $L^2_{uloc} (\R_+;V')$
made of all strongly measurable vector fields ${\bf f}:\ \R_+\to V'$ such that
\begin{equation*}
   \|\bff\|_{L^{2}_{uloc}(\R_+;V')}
   :=\left[\sup_{t\geq0}\int_{t}^{t+1}\|\bff(s)\|_{V'}^{2}\, ds \right]^{1/2}<+\infty.
 \end{equation*}
 We will see in the following, that the above space, which strictly contains both
 $L^2(\R_+;V')$ and $L^\infty(\R_+;V')$ is well suited for our framework. We will prove
 the following existence result, in order to make the paper self-contained.
\begin{theorem}
  \label{thm:existence_temps_long} 
  Let $\vv_0 \in H$, and let ${\bf f} \in L^2_{uloc}(\R_+;V')$. Then, there exists a weak
  solution $\vv$ to the NSE~(\ref{eq:NSE}) global-in-time, obtained by Galerkin
  approximations, such that
 \begin{equation*}
\vv \in L^2_{loc} (\R_+;V) \cap L^\infty(\R_+;H),
\end{equation*}
 and which satisfies for all $t \ge 0$,
 \begin{align} 
   &  \label{eq:main_estimate_1} \zoom  \| \vv (t) \|^2 \le \| \vv_0 \|^2 + \left ( 3 + \frac{C_\Om}{\nu}  \right )
   \frac{ {\cal F}^2}{\nu }, 
   \\
   &  \label{eq:main_estimate_2}  \nu \int_0^t \| \g \vv (s) \|^2 ds \le \| \vv_0 \|^2 + ([t] + 1 ) \frac{{\cal
     F}^2}{\nu},
 \end{align}
where $\mathcal{F}:=\|\bff\|_{L^{2}_{uloc}(\R_+;V')}$.
\end{theorem}
\begin{remark}
  The weak solution $\bv$ shares most of the properties of the Leray-Hopf weak solutions,
  with estimates valid for all positive times.  Notice that we do not know whether or not
  this solution is unique. Anyway, it will not get \say{regular} as $t \to+\infty$, which
  is the feature of interest for our study.  As usual by regular we mean that it does not
  necessarily have the $L^2$-norm of the gradient (locally) bounded, hence that it is not
  a strong solution.
\end{remark}
%
\subsection{Long-time averaging}
This section is devoted to state the main results of the paper about long-time and
ensemble averages. 
\begin{theorem}
  \label{thm:main_theorem}
  Let be given $\vv_0 \in H$, $\bff\in L^{2}_{uloc}(\R_+;V')$, and let $\vv$ a
  global-in-time weak solution to the NSE~\eqref{eq:NSE}.  Then, there exist
\begin{enumerate}[a)]
\item a sequence $\{t_n\}_{n\in\N}$ such that $\lim\limits_{n \to \infty }t_n=+\infty$;
\item a vector field $\mean{\bv}\in V$;
\item vector field $\overline \bff \in V'$;
\item a vector field $\bB\in L^{3/2}(\Omega)^3$;
\item a second order tensor field $\reyn \in L^3(\Omega)^9$;
\end{enumerate} 
 such that it holds:
\begin{enumerate}[i)]
\item when $n \to \infty$,
    \begin{equation*}
      \begin{aligned}
        &M_{t_n}(\bv)\rightharpoonup
        \mean{\bv}\qquad\text{in }V,
        \\
        & M_{t_n} (\bff) \rightharpoonup
        \mean{\bff} \qquad\text{in }V',
        \\
        & M_{t_{n}}\big ((\bv\cdot\nabla)\,\bv\big )\rightharpoonup\bB
        \qquad \text{in } L^{3/2}(\Omega)^3,
        \\
        &M_{t_n}(\bv{'} \otimes \bv{'} )\rightharpoonup
        \reyn\qquad\text{in }L^3(\Omega)^9;\\
        & M_{t_n} ( < \fv, \vv > ) \to \,  <\overline \fv, \vm > + \overline{ < \fv', \vv' > },
      \end{aligned}
    \end{equation*}
    where $\vv' = \vv- \vm$, and $\fv' = \fv - \overline \fv$;
  \item the closed Reynolds equations~\eqref{eq:R} holds true in the weak sense: 
  \begin{equation}
  \label{eq:R}
  \left\{
    \begin{aligned}
     (\mean{\bv}\cdot\nabla)\, \mean{\bv} -\nu\Delta \mean{\bv}+\nabla
      \mean{p}+\nabla\cdot\reyn&=\mean{\bff}\qquad &\text{ in }\Omega,
      \\
      \nabla\cdot \mean{\bv}&=0\qquad &\text{ in }\Omega,
      \\
      \mean{\bv}&=\mathbf{0}\qquad &\text{ on }\Gamma;
    \end{aligned}  
  \right.
\end{equation}
  \item the following equalities $\bF = \bB - (\vm \cdot \g )\, \vm = \g \cdot
    \reyn$ are valid in $\mathcal{D}'(\Omega)$;
  \item the following energy balance holds true
  \begin{equation*}
    \nu\|\nabla\mean{\bv}\|^2+\int_\Omega
    \bF\cdot\mean{\bv}\,d\bx= \, <\mean{\bff},\mean{\bv} >;
  \end{equation*}
\item the turbulent dissipation $\E$ is bounded by the sum of the work of the Reynolds
  stress on the mean flow and the external turbulent fluxes,
\begin{equation}
  \label{eq:reyn_dissipative2} 
\E =  \nu  \mean{\|\nabla{\bv}'\|^2}\le \int_\Omega(\nabla\cdot
  \reyn)\cdot\mean{\bv}\,d\bx + \overline {< \fv', \vv' > };
\end{equation}
\item \label{it:foias_generalized}  if in addition the source term $\bff$ verifies:
\begin{equation}
  \label{eq:additional_ass_f}
  \exists \  \widetilde {\bff}\in V', \quad \text{such that} \quad
  \lim_{t\to+\infty}\int_{t}^{t+1}\|\bff(s)-\widetilde {\bff}\|_{V'}^2\,ds=0,
\end{equation}
then $\overline \bff = \widetilde {\bff}$ and $\overline {< \fv', \vv' > } = 0$; in
particular, the Reynolds stress $\reyn$ is dissipative in average, that
is~(\ref{eq:reyn_dissipative}) holds true.
\end{enumerate}
\end{theorem}

%
Our second result has to be compared with results in~Layton~\textit{et al.}~\cite{LL2016,
  Lay2014}, where the long-time averages are taken for an ensemble of solutions.
\begin{theorem}
  \label{thm:main_ensemble}
  Let be given a sequence $\suite \bff k\subset L^q(\Omega)$ converging weakly to some
  $\smean{\bff}$ in $L^q(\Omega)$, with $q>\frac{6}{5}$ and let $\{\mean{\bv}^ k\}_{k \in
    \N} $ be the associated long-time average of velocities, whose existence has been
  proved in Theorem~\ref{thm:main_theorem}. Then, the sequence of arithmetic averages of
  the long-time limits $\{\smean{\bv}^n\}_{n\in\N}$, defined as
\begin{equation*}
  \smean{\bv}^n :=  \frac{1}{n}\sum_{k=1}^n\mean{\bv}^k,
\end{equation*}
converges weakly, as $n\to+\infty$, in $V$ to some $\smean{\bv}$, which satisfies the
following system of Reynolds type
\begin{equation*}
  \left\{
    \begin{aligned}
      \big(\smean{\bv}\cdot\nabla\big)\, \smean{\bv} -\nu\Delta \smean{\bv}+\nabla
      \smean{p}+\nabla\cdot\smean{\reyn}&=\smean{\bff}\qquad &\text{ in }\Omega,
      \\
      \nabla\cdot \smean{\bv}&=0\qquad &\text{ in }\Omega,
      \\
      \smean{\bv}&=\mathbf{0}\qquad &\text{ on }\Gamma,
  \end{aligned}  
\right.
\end{equation*}
where $\smean{\reyn}$ is dissipative in average, that is more precisely
\begin{equation*}
  0 \le\frac{1}{|\Om|} \int_\Omega\big(\nabla\cdot\smean{\reyn}\big)\cdot\smean{\bv}\,d\bx.  
\end{equation*}
\end{theorem}
In this case we do not have a sharp lower bound on the dissipation as in
Theorem~\ref{thm:main_theorem}, since here the averaging is completely different and the
fluctuations are not those emerging in long-time averaging. Nevertheless, the main
statement is in the same spirit of the first proved result.
\section{Navier-Stokes equations with uniformly-local source terms}
\label{sec:NSE-estimates}
This section is devoted to sketch a proof of Theorem~\ref{thm:existence_temps_long}.  Most
of the arguments are quite standard and we will give appropriate references at each step,
to focus on what seems (at least to us) non-standard when ${\bf f} \in
L^2_{uloc}(\R_+;V')$; especially the proof of the uniform
$L^2$-estimate~\eqref{eq:energy_1}, which is the building block for the results of the
present paper. Before doing this, we introduce the function spaces we will use, and
precisely define the notion of weak solutions we will deal with.
\subsection{Functional setting}
\label{sec:notation}
Let $\Omega\subset \R^3$ be a bounded open set with Lipschitz boundary $\partial
\Omega$. This is a sort of minimal assumption of regularity on the domain, in order to
have the usual properties for Sobolev spaces and to characterize in a proper way
divergence-free vector fields in the context of Sobolev spaces, see for instance
Constantin and Foias~\cite{CF1988}, Galdi~\cite{Gal2000a,Gal2011}, Girault and
Raviart~\cite{GR86}, Tartar~\cite{LT06}.

We use the customary Lebesgue spaces $(L^p(\Omega),\|\,.\,\|_p)$ and Sobolev spaces
$(W^{1,p}(\Omega),\|\,.\,\|_{1,p})$.  For simplicity, we denote the $L^2$-norm simply by
$\|\,.\,\|$ and we write $H^{1}(\Omega):=W^{1,2}(\Omega)$. For a given sequence
$\{x_n\}_{n\in\N}\subset X$, where $(X,\|\,.\,\|_X)$ is Banach space, we denote by
$x_n\rightarrow x$ the strong convergence, while by $x_n\rightharpoonup x$ the weak one.

As usual in mathematical fluid dynamics, we use the following spaces
\begin{equation*}
  \begin{aligned}
  & {\cal V} = \{ \boldsymbol{\varphi} \in {\cal D}(\Om)^3, \, \, \g \cdot
  \boldsymbol{\varphi} = 0 \}, 
  \\
    &H =
    \left\{\bv\in L^2(\Omega)^3, \ \nabla\cdot \bv=0,\
      \bv\cdot\mathbf{n}=0 \text{ on }\Gamma\right\},
    \\
    &V = 
    \left\{\bv\in H^1_{0}(\Omega)^3, \,  \nabla\cdot \bv=0\right\},
  \end{aligned}
\end{equation*}
and we recall that ${\cal V}$ is dense in $H$ and $V$ for their respective
topologies~\cite{GR86, LT06}.

Let $(X,\|\,.\,\|_X)$ be a Banach space, we use the Bochner spaces $L^p(I;X)$, for $I=
[0,T]$ (for some  $T>0$) or $I=\R_+$ equipped with the norm
\begin{equation*}
\| u \|_{L^p(I; X)} := \left\{
    \begin{aligned}
      &\left( \int_I \| u(s) \|^p_X ds \right )^{\frac{1}{p}}\qquad \textrm{ for }1\leq p<\infty,
      \\
      &\textrm{ess sup}_{s\in I} \| u(s) \|_X\qquad \textrm{ for }p=+\infty.
    \end{aligned}
\right.\end{equation*}
 
The existence of weak solutions for the NSE~(\ref{eq:NSE}) is generally proved in the
literature when $\vv_0 \in H$ and the source term ${\bf f} \in L^2(I; V')$, or
alternatively when the source term is a given constant element of $V'$.  In order to study
the long-time behavior of weak solutions of the NSE~(\ref{eq:NSE}), we aim to enlarge the
class of function spaces allowed for the source term $\fv$, to catch a more complex
behavior than that coming from constant external forces, as initially developed
in~\cite{Lew2015}. To do so, we deal with ``uniformly-local'' spaces, as defined below in
the most general setting.
\begin{definition}
  Let be given $p\in[1,+\infty[$. We define $L^{p}_{uloc}(\R_+;X)$ as the space of
  strongly measurable functions $f:\R_+\to X$ such that
\begin{equation*}
\|f\|_{L^{p}_{uloc}(X)} :=\left[\sup_{t\geq0}\int_{t}^{t+1}\|f(s)\|_{X}^{p} ds \right]^{1/p}<+\infty.
    \end{equation*}
\end{definition}
It is easily checked that the spaces $L^p_{uloc}(\R_{+};X)$ are Banach spaces strictly containing
both the constant $X$-valued functions, and also $L^p(\R_+; X)$, as illustrated
by the following elementary lemma.
\begin{lemma}
  Let be given $f\in C(\R_+; X)$ converging to a limit $\ell \in X$, when
  $t\to+\infty$. Then, for any $p\in[1,+\infty[$, we have that $f \in
  L^{p}_{uloc}(\R_{+};X)$, and there exists $T>0$ such that
%
\begin{equation*}
  \|f\|_{L^{p}_{uloc}(X)} \le \left [\sup_{ t \in [0, T+1]} \|f (t)\|_X^p + 2^{p-1}(1+\|\ell\|_X)^p \right ]^{\frac{1}{p}} .  
\end{equation*}
\end{lemma}
\begin{proof} 
  As $\zoom \ell = \lim_{t \to+\infty} f(t)$, there exists $T >0$ such that:
  $\forall\,t>T$, $\|f(t) -\ell\|_X \le 1$.  In particular, it holds
\begin{equation*}
  \int_t^{t+1}  \| f(s) \|^p_X\, ds \le 2^{p-1}(1+ \| \ell \|_X)^p\qquad \text{for }t>T,
\end{equation*}
while for all $t \in[0, T]$, 
\begin{equation*}
  \int_t^{t+1}  \| f(s) \|^p_{X}\,ds  \le \sup_{t \in [0, T+1] } \| f(t) \|_X^p,
\end{equation*}
hence the result.
\end{proof} 
However, it easy to find examples of discontinuous functions in $L^{p}_{uloc}(\R_+;X)$
which are not converging when $t \to+\infty$, and which are not belonging to
$L^p(\R_+;X)$.
\subsection{Weak solutions}
There are many ways of defining weak solutions to the NSE (see also
P.-L.~Lions~\cite{PL96}). Since we are considering the incompressible case, the pressure
is treated as a Lagrange multiplier. Following the pioneering idea developed by
J.~Leray~\cite{Ler1934}, the NSE are projected over spaces of divergence-free
functions. This is why when we talk about weak solutions the NSE, only the velocity $\vv$
is mentioned, not the pressure.

As in J.-L.~Lions~\cite{JLL69}, we give the following definition of weak solution, see
also Temam~\cite[Ch.~III]{Tem1977b}.
\begin{definition}[Weak solution]
  Given $\bv_{0}\in H$ and ${\bf f} \in L^2(I; V')$ we say that $\vv$ is a weak solution
  over the interval $I = [0,T]$ if the following items are fulfilled:
  \begin{enumerate}[i)]
  \item the vector field $\vv$ has the following regularity properties
    \begin{equation*}
      \vv \in L^2 (I; V) \cap L^\infty(I; H),
    \end{equation*}
    and is weakly continuous from $I$ to $H$, while
    $\lim\limits_{t\to0^{+}}\|\vv(t)-\vv_{0}\|_{H}=0$;
  \item for all $\boldsymbol{\varphi} \in {\cal V}$,
    \begin{equation*}
      \begin{aligned}
        \frac{d}{dt} \int_\Om \vv (t, \x) \cdot \boldsymbol{\varphi}(\x) \,d\x &- \int_\Om
        \vv(t, \x) \otimes \vv(t, \x) : \g \boldsymbol{\varphi}(\x)\,d\x
        \\
        & \zoom + \nu \int_\Om \g \vv (t, \x) : \g \boldsymbol{\varphi}(\x)\,d\x =\ 
        <{\bf f}(t), \boldsymbol{\varphi} > ,
      \end{aligned}
    \end{equation*}
    holds true in $\mathcal{D}'(I)$;
  \item the energy inequality
    \begin{equation}
      \label{eq:energy}
      \frac{d}{dt}\frac{1}{2}\|\bv(t)\|^{2}+\nu\|\nabla\bv(t)\|^{2}\leq
      \ <\bff(t),\bv(t) > ,
    \end{equation}
    holds in ${\cal D}'(I)$, where we write $\vv(t)$ instead of $\vv (t, \cdot)$ for
    simplicity.
  \end{enumerate}
  When $\bff\in L^{2}(0,T;V')$ and $\bv$ is a weak solution in $I=[0,T]$, and this holds
  true for all $T>0$, we speak of a \say{global-in-time solution}, or simply a \say{global
    solution}. In particular, $ii)$ is satisfied in the sense of  $\mathcal{D}'(0,+\infty)$.
\end{definition}

There are several ways to prove the existence of (at least) a weak solution to the NSE.
Among them, in what follows, we will use the Faedo-Galerkin method. Roughly speaking, let
$\{\boldsymbol {\varphi}_n\}_{n\in\N}$ denote a Hilbert basis of $V$, and let, for
$n\in\N$, $V_n := \hbox{span}\{\boldsymbol{\varphi}_1,\cdots,\boldsymbol{\varphi}_n\}$. By
assuming ${\bf f} \in L^2(I; V')$, it can be proved by the Cauchy-Lipschitz theorem
(see~\cite{JLL69}) the existence of a unique $\vv_n \in C^1(I; V_n)$ such that for all
$\boldsymbol{\varphi}_k$, with $k= 1, \dots, n$ it holds
\begin{equation}
  \label{eq:galer1} 
  \begin{aligned}
    \zoom \frac{d}{dt} \int_\Om \vv_n (t, \x) \cdot \boldsymbol{\varphi}_k(\x)\,d\x &-
    \int_\Om \vv_n(t, \x) \otimes \vv_n(t, \x) : \g \boldsymbol{\varphi}_k(\x)\,d\x
    \\
    & \zoom + \nu \int_\Om \g \vv_n (t, \x) : \g \boldsymbol{\varphi}_k(\x)\,d\x = \ <{\bf
      f}(t), \boldsymbol{\varphi}_k >,
\end{aligned}
\end{equation}
 and which naturally satisfies the energy balance (equality)
 \begin{equation}
  \label{eq:energy_n}
  \frac{d}{dt}\frac{1}{2}\|\bv_n(t)\|^{2}+\nu\|\nabla\bv_n(t)\|^{2}=\ <\bff,\bv_n >.
\end{equation}
It can be also proved (always see again~\cite{JLL69}) that from the sequence $\suite \vv
n$ one can extract a sub-sequence converging, in an appropriate sense, to a weak solution
to the NSE.  When $I = \R_+$ we get a global solution.

However, if assume ${\bf f} \in L^2_{uloc}(\R_+;V')$, the global result of existence does
not work so straightforward. Of course, for any given $T>0$, we have
\begin{equation*}
  L^2_{uloc}(\R_+; V')_{\big\vert \zoom{[0,T]}}  \hookrightarrow  L^2([0,T]; V'),
\end{equation*}
where $L^2_{uloc}(\R_+; V')_{\vert {[0,T]}}$ denotes the restriction of a function in
$L^2_{uloc}(\R_+; V')$ to $[0,T]$. Therefore, no doubt that the construction above holds
over any time-interval $[0,T]$. In such case letting $T$ go to $+\infty$ to get a global
solution (with some uniform control of the kinetic energy) is not obvious, and we do not
know any reference explicitly dealing with this issue, which deserves to be investigated
more carefully.  This is the aim of the next subsection, where we prove the most relevant
a-priori estimates.
\subsection{A priori estimates}
Let be given ${\bf f} \in L^2_{uloc}(\R_+; V')$, and let $\vv_n = \vv$ be the solution of
the Galerkin projection of the NSE over the finite dimensional space $V_n$. The function
$\vv$ satisfies~(\ref{eq:galer1}) and~(\ref{eq:energy_n}) (we do not write the subscript
$n\in\N$ for simplicity), is smooth, unique, and can be constructed by the Cauchy-Lipschitz
principle over any finite time interval $[0,T]$. Hence, we observe that by uniqueness it
can be extended to $\R_+$. We then denote ${\cal F} := \| {\bf f} \|_{ L^2_{uloc} (\R_+;
  V') }$ and then after a delicate manipulation of the energy balance combined with the
Poincar\'e inequality, we get the following lemma.
\begin{lemma}
\label{lem:energy-estimates}
For all $t \ge 0$ we have
\begin{equation}
  \label{eq:energy_1}
  \|\bv(t)\|^2\leq\|\bv_0\|^2+\left(3+\frac{C_\Om}{\nu}\right)\frac{\mathcal{F}^{2}}{\nu},
\end{equation}
as well as
\begin{equation}
  \label{eq:dissipation} 
  \nu  \int_0^t\|\nabla \bv(s)\|^2\,ds \leq \|\bv_0\|^2+({[t]+1})\frac{\mathcal{F}^2}{\nu},
\end{equation}
where $C_\Om$ denotes the  constant in the Poincar\'e inequality $\|\bu\|^2\leq
C_\Omega\|\nabla \bu\|^2$, valid for all $\bu\in V$.
\end{lemma}
\begin{proof}
  We focus on the proof of the a priori estimate~\eqref{eq:energy_1}, the
  estimate~\eqref{eq:dissipation} being a direct consequence of the energy balance.
   By the Young inequality we deduce from the the energy inequality, 
   \begin{equation*}
     \begin{array}{c} 
       \forall \, \xi, \tau\in\R_{+} \quad \hbox{s.t.} \quad 0 \le \xi \le \tau,  \\~  \\ 
       \zoom 
       \|\bv(\tau)\|^{2}+
       \nu\int_{\xi}^{\tau}\|\nabla\bv(s)\|^{2}\,ds\leq
       \|\bv(\xi)\|^{2}+
       \frac{1}{\nu}\int_{\xi}^{\tau}\|\bff(s)\|^{2}_{V'}\,ds.  \end{array}
   \end{equation*}
In particular, when $0 \le \tau - \xi \le 1$,  
\begin{equation}
\label{eq:estimates_basic}
    \|\bv(\tau)\|^{2}+
    \nu\int_{\xi}^{\tau}\|\nabla\bv(s)\|^{2}\,ds\leq
  \|\bv(\xi)\|^{2}+
  \frac{\mathcal{F}^{2}}{\nu}.
\end{equation}
From this point, we argue step by step. The case $0 \le t \le 1$ is the first step, which
is straightforward. The second step is the heart of the proof.  The issue is that energy
may increase, without control, when the time increases.

We will show that even if this happens, we can still keep the control on it, thanks
to~\eqref{eq:estimates_basic}. The last step is the concluding step, carried out by
induction on $n$ writing $ t = \tau +n$, for $\tau \in [0,1]$.

\medskip 

\noindent \underline {\sc Step 1}.  $t \in [0,1]$: take  $\xi= 0$,  $t = \tau \in [0,1]$. Then 
  \begin{equation*}
\zoom \| \vv (t) \|^2 \le \| \vv_0 \|^2 + \frac{\mathcal{F}^{2}}{  \nu}\qquad\forall\,t\in[0,1]. 
\end{equation*}

\noindent \underline {\sc Step 2}.  We show that the following implication holds true

\begin{equation*}
  \zoom \|  \vv (t ) \| \le \| \vv(t+1 ) \| \, \Rightarrow \, 
  \left \{ \begin{array}{l}  \zoom  \| \vv(t) \| \le
      \left(1+\frac{C_\Om}{\nu}\right)\frac{\mathcal{F}^{2}}{\nu}, 
      \\ 
      \zoom
      \| \vv(t+1) \| \le \left(2+\frac{C_\Om}{\nu}\right)\frac{\mathcal{F}^{2}}{\nu}. 
    \end{array} \right.
\end{equation*}

In the following we will set 
\begin{equation}
  \label{eq:notation_constant} 
  \zoom \mathcal{C}^2:= \left
  ( \| \vv_0 \|^2 + \frac{\mathcal{F}^{2}}{\nu} \right ) +
\left(2+\frac{C_\Om}{\nu}\right)\frac{\mathcal{F}^{2}}{\nu} = \|
\vv_0\|^2+\left(3+\frac{C_\Om}{\nu}\right)\frac{\mathcal{F}^{2}}{\nu}.
\end{equation}

\noindent \underline {\sc Sub-Step 2.1.} By the energy inequality with $\xi = t$, and
$\tau = t+1$, we have, by using the hypothesis on the $L^2$-norm at times $t$ and $t+1$: 
\begin{equation*}
    \begin{aligned}
      \nu\int_{t}^{t+1}\|\nabla\bv(s)\|^{2}\,ds
    &\leq \|\bv(t+1)\|^{2}-
    \|\bv(t)\|^{2}+\nu\int_{t}^{t+1}\|\nabla\bv(s)\|^{2}\,ds
\leq \frac{\mathcal{F}^{2}}{\nu}.
  \end{aligned}
\end{equation*} 
Hence, by the Poincar\'e's inequality:

\begin{equation}
  \label{eq:substep_2_1} 
  \zoom
  \int_{t}^{t+1}\|\bv(s)\|^{2}\,ds \le    C_\Om
  \int_{t}^{t+1}\|\nabla\bv(s)\|^{2}\,ds \le   \frac{C_\Om \mathcal{F}^{2}}{\nu^2},
\end{equation}

\noindent \underline {\sc Sub-Step 2.2.}  Let be given $\epsilon>0$ and let $\xi\in
[t,t+1]$ be such that
 \begin{equation*}
  \|\bv(\xi)\|^{2}<\inf_{s\in[t,t+1]}\|\bv(s)\|^{2}+\epsilon\leq\|\bv(s)\|^{2}+\epsilon  
  \qquad\forall\,s\in [t,t+1].
\end{equation*}
Let us write:
 \begin{equation*}
  \begin{aligned}
    \|\bv(t)\|^{2}\leq \|\bv(t+1)\|^{2}&=\|\bv(t+1)\|^{2}- \|\bv(\xi)\|^{2} +
    \|\bv(\xi)\|^{2}
    \\
    &=  \|\bv(t+1)\|^{2}- \|\bv(\xi)\|^{2} + \int_{t}^{t+1}
    \|\bv(\xi)\|^{2}\,ds,
  \end{aligned}
\end{equation*}
being the integration with respect to the $s$ variable.

To estimate the right-hand side we use the energy inequality with $\tau=t+1$ to get
\begin{equation*}
  \|\bv(t+1)\|^{2}- \|\bv(\xi)\|^{2}  \le \frac{\mathcal{F}^2}{\nu }.
\end{equation*}
Moreover, using the estimate~\eqref{eq:substep_2_1} we get,
 $$ \int_{t}^{t+1}
    \|\bv(\xi)\|^{2}\,ds \le \int_{t}^{t+1} (
    \|\bv(s)\|^{2}+\epsilon)\,ds \le \frac{C_\Om \mathcal{F}^{2}}{\nu^2} + \E,$$
    therefore, letting $\E \to 0$ yields, 
 \begin{equation*}
  \|\bv(t)\|^{2}\leq\left(1+\frac{C_\Om}{\nu}\right)\frac{\mathcal{F}^{2}}{\nu}.
\end{equation*}
In addition, we get from the energy inequality
\begin{equation*}
  \|\bv(t+1 )\|^{2}\leq \| \vv(t) \|^2 + \frac{\mathcal{F}^{2}}{\nu} \le
  \left(2+\frac{C_\Om}{\nu}\right)\frac{\mathcal{F}^{2}}{\nu}. 
\end{equation*}

\noindent \underline {\sc Step 3}. Conclusion of the proof of \eqref{eq:energy_1}. Any
$t\geq0$ can be decomposed as 
\begin{equation*}
t =n  + \tau, \quad\text{with }n\in\N\text{ and } \tau \in [0,1[.
\end{equation*}

We argue by induction on $n$.  If $n = 0$, estimate~\eqref{eq:energy_1} has been proved
in Step~1. Assume that~\eqref{eq:energy_1} is satisfied for  $t:=n+\tau$, for all $n\leq N$ that
is, 
\begin{equation*}
\| \vv (n + \tau ) \|\le \mathcal{C}\qquad n=0,\dots,N,
\end{equation*}
where the constant $\mathcal{C}$ is defined
by~\eqref{eq:notation_constant}.

\noindent If $ \| \vv (N +1+ \tau ) \|< \| \vv (N + \tau ) \|$, then~\eqref{eq:energy_1}
holds at the time $t=N+1+\tau$, by the inductive hypothesis.

\noindent  If   $ \| \vv (N + \tau ) \|\le  \| \vv (N + 1+ \tau ) \|$, then the
inequality~\eqref{eq:energy_1} is satisfied by  Step~2 for $t=N+1+\tau$, ending the proof.

\end{proof}
Once we have proved that the uniform (independent of $n\in\N$) $L^2$-estimate is satisfied
by the Galerkin approximate functions, it is rather classical to prove that we can extract
a sub-sequence that converges weakly$*$ in $L^\infty(0,T;L^2(\Omega))$ (for all positive
$T$) to a weak solution to the NSE, which inherits the same bound. We refer to the
references already mentioned for this point.
\section{Reynolds decomposition and time-averaging }
\label{sec:reynolds_rules}
We sketch the standard routine, concerning time-averaging, when used in turbulence
modeling practice.  In particular, we recall the Reynolds decomposition and the Reynolds
rules. Then, we give a few technical properties of the time-averaging operator $M_t$, defined by
\begin{equation*}
M_t (\psi) := \frac{1}{t}\int_0^t \psi (s) \,ds,
\end{equation*}
for a given fixed time $t > 0$.  We need to apply it not only to real functions of a real
variable, but also to Banach valued functions, hence we need to deal with the Bochner integral.

Before all, we start with the following corollary of
Bochner theorem (see Yosida~\cite[p.~132]{Yos1980}).
\begin{lemma}
  \label{lem:bochner} 
  Assume that, for some $t>0$ we have $\psi\in L^p([0,t]; X)$ (namely $\psi$ is a Bochner
  $p$-summable function over $[0,t]$, with values in the Banach space $X$). Then, it holds 
  \begin{equation}
    \label{eq:1}
    \|M_t(\psi)\|_X\leq \frac{1}{t^{\frac{1}{p}}}\|\psi\|_{L^p([0,t];X)}.
  \end{equation}
\end{lemma}
\noindent Estimate~\eqref{eq:1} is the building block to give a sense to the long-time
average as 
\begin{equation}
  \label{eq:time-filtering-def}
\overline \psi := \lim_{t \to \infty} M_t (\psi),
\end{equation}
whenever the limit exists.

It is worth noting at this stage that the mapping $\mu$ behind the long-time average,
defined on the Borel sets of $\R_+$ by
 \begin{equation*}
   A\mapsto\lim_{t\to+\infty}\frac{1}{t}\lambda(A\cap[0,t])
   :=\lim_{t\to+\infty} M_t ( \mathbbm{1}_A ) = \mu(A),
 \end{equation*}
 where $\lambda$ the Lebesgue measure, 
 is not --strictly speaking-- a probability measure since it is not
 $\sigma$-additive\footnote{The mapping $\mu$ satisfies $\mu(A\cup B) = \mu(A) + \mu (B)$
   for $A \cap B = \emptyset$ but, on the other hand, we have $\sum_{n=0}^\infty \mu([n,
   n+1[) = 0 \not= 1 = \mu ( \bigcup_{n=0}^\infty [n, n+1[ )$. }. Therefore, the quantity
 $\mean{\psi}$ is not rigorously a statistic, even if practitioners could be tempted to
 write it (in a suggestive and evocative meaningful way) as follows:
 \begin{equation*}
 \mean{\psi}(\bx) = \int_{\R_+} \psi(s, \x)\, d \mu(s).
\end{equation*}
\subsection{General setup of turbulence modeling}
We recall that $M_t$ is a linear filtering operator which commutes with differentiation
with respect to the space variables (the so called Reynolds rules). In particular, one has
the following result (its proof is straightforward), which is essential for our modeling
process.
\begin{lemma}
  Let be given $\psi \in L^1([0,T], W^{1,p}(\Omega))$, then 
  \begin{equation*}
    D M_t(\psi)=   M_t(D\psi)\qquad
    \forall\,t >0,
  \end{equation*}
  for any first order differential operator $D$ acting on the space variables
  $\bx\in\Omega$. 
\end{lemma}
By denoting the long-time average of any field $\psi$ by $\mean{\psi}$ as
in~\eqref{eq:time-filtering-def}, we consider the fluctuations $\psi'$ around the mean
value, given by the Reynolds decomposition
\begin{equation*}  
\psi :=\mean{\psi}+\psi'. 
\end{equation*} 
Observe that long-time averaging has many convenient {\it formal} mathematical properties, recalled in the following.
\begin{lemma} 
  The following formal properties holds true, provided the long-time averages do exist.
  \begin{enumerate}
  \item The \say{bar operator} preserves the no-slip boundary condition. In other words,
    if $\psi_{ \vert_\Gamma}= 0$, then $\mean{\psi }_{ \vert_\Gamma} =0$;
  \item Fluctuation are in the kernel of the bar operator, that is $\mean{\psi'}=0$;
  \item The bar operator is idempotent, that is $\mean{\mean{\psi }}=\mean{\psi }$, which also
    yields $\mean{\mean{\psi }\,\varphi}=\mean{\psi }\mean{\varphi}$.
  \end{enumerate}
\end{lemma}
Accordingly, the velocity components can be decomposed in the Reynolds
decomposition as follows:
\begin{equation*}
  \bv(t, \bx) = \mean{\bv}(\bx) + \bv'(t, \bx).
\end{equation*}
Let us determine (at least formally) the equation  satisfied by $\vm$. To do so, we use the
above Reynolds rules to expand the nonlinear quadratic term into
\begin{equation}
 \label{eq:decomposition_Reynolds_stress} 
\mean{\bv\otimes \bv} =
\mean{\bv}\otimes\mean{\bv}+\mean{\bv'\otimes\bv'}, 
\end{equation} 
which follows by observing that $\overline {\vv' \otimes \vm} = \overline {\vm \otimes
  \vv'} = \mathbf{0}$.  

The above rules allow us to prove the following result showing a certain ``orthogonality''
between averages and fluctuations.
\begin{lemma}
  \label{lem:lemma-orthogonality}
  Let be given a linear space $X\subseteq L^{2}(\Omega)$ with a scalar product
$(\,.\,,\,.\,)$. Let in addition be given a function $\psi:\,\R^+\to X$ such that
$\mean{\psi}$ is well defined. Then it follows that
  \begin{equation*}
\mean{    (\psi',\psi')}=\mean{(\psi,\psi)}-(\mean{\psi},\mean{\psi}),
  \end{equation*}
provided all averages are well defined.
\end{lemma}
\begin{proof}
  The proof follows by observing that $\psi'=\psi-\mean{\psi}$, hence 
  \begin{equation*}
    \mean{    (\psi',\psi')}=\mean{
      (\psi-\mean{\psi},\psi-\mean{\psi})}=\mean{(\psi,\psi)}-
    2\mean{(\psi,\mean{\psi})}+\mean{(\mean{\psi},\mean{\psi})},  
  \end{equation*}
and by the Reynolds rules $\mean{(\psi,\mean{\psi})}=(\mean{\psi},\mean{\psi})$ and
$\mean{(\mean{\psi},\mean{\psi})}=(\mean{\psi},\mean{\psi})$,
from which is follows the thesis.

In particular, we will use it for the $V$ scalar product showing that 
\begin{equation}
  \label{eq:new_Foias}
 \mean{ \|\nabla\bu'\|^2}=\mean{\|\nabla\bu\|^2}-\|\nabla\mean{\bu}\|^2,
\end{equation}
for  $\bu:\,\R^+\to V$, such that the long-time average exists. 
\end{proof}
\begin{remark}
  Observe that, for weak solutions of the NSE $\bv$, the average
  $M_{t}(\|\nabla\bv\|^{2})$ is bounded uniformly, by the result of
  Theorem~\ref{thm:existence_temps_long} and --up to sub-sequences-- some limit can be
  identified. Moreover, by using an argument similar to Lemma~\ref{lem:bochner} it follows
  that
  \begin{equation*}
    \|M_{t}(\nabla\bv)\|^2\leq M_{t}(\|\nabla\bv\|^2),
  \end{equation*}
  which show that (up to sub-sequences) also the second term from the right-hand side can
  be properly defined. Consequently, also the average of the squared $V$-norm of the
  fluctuations from the left-hand side is well defined by difference.
\end{remark}
Long-time averaging applied to the Navier-Stokes equations (in a strong
formulation) gives the following ``equilibrium problem'' for the long-time average
$\mean{\bv}(\bx)$,
\begin{equation}
 \label{eq:Reyn_eq_primitive} 
  \left\{  
    \begin{aligned}
      -\nu\Delta
      \mean{\bv}+\nabla\cdot(\mean{\bv\otimes\bv})+\nabla\mean{p}&=\mean{\bff}
      \qquad \text{in }\Omega,
      \\
      \nabla\cdot \mean{\bv}&=0\qquad \text{in }\Omega, 
      \\
      \mean{\bv}&=\mathbf{0}\qquad \text{on }\Gamma,
    \end{aligned}
  \right.
\end{equation}
which we will treat in the next section to make appear the closure problem. 
\subsection{Reynolds stress and Reynolds tensor}
The first equation of system~\eqref{eq:Reyn_eq_primitive} can be rewritten also as follows
(by using the decomposition into averages and fluctuations)
\begin{equation}
\label{eq:Reynolds_eqs} 
  -\nu\Delta
  \mean{\bv}+\nabla\cdot(\mean{\bv}\otimes\mean{\bv})+\nabla\mean{p}=
  -\nabla\cdot(\mean{\bv'\otimes\bv'})+\mean{\bff},
\end{equation}
called the Reynolds equations.  Beside convergence issues, a relevant point is to
characterize the average of product of fluctuations from the right-hand side, which is the
divergence of the so called Reynolds stress tensor, defined as follows
\begin{equation} 
  \label{eq:reynolds_stress_1}
  \reyn =\mean{\bv'\otimes\bv'}.
\end{equation} 
The Boussinesq hypothesis, formalized in~\cite{Bou18777} (see also~\cite[Ch.~3 \&
4]{CL2014}, for a comprehensive and modern presentation) corresponds then to a closure 
hypothesis with the following linear constitutive equation:
\begin{equation} 
  \label{eq:reynolds_stress} 
  \reyn=-\nu_t\frac{\nabla \mean{\bv}+\nabla\mean{\bv}^T}{2} + \frac{2}{3} k\, {\rm Id},  
\end{equation} 
where $\nu_t\geq0$ is a scalar coefficient, called turbulent viscosity
or eddy-viscosity (sometimes called ``effective viscosity''), and
\begin{equation*}
  k = \frac{1}{2} \overline {| \vv' |^2},
\end{equation*}
is the turbulent kinetic energy,
see~\cite{BIL2006,CL2014}. Formula~\eqref{eq:reynolds_stress} is a linear relation between
stress and strain tensors, and shares common formal points with the linear constitutive
equation valid for Newtonian fluids.  In particular, this assumptions motivates the fact
that $\reyn$ must be dissipative\footnote{The sign adopted in~\eqref{eq:reynolds_stress_1}
  is a convention consistent with our mathematical approach. However, according to the
  analogy of the Reynolds stress with viscous forces, it is also common to define it as $
  \reyn :=- \mean{\bv'\otimes\bv'}$, which does not change anything.} on the mean
flow. Some recent results in the numerical verification of the hypothesis can be found in
the special issue~\cite{Bois2007} dedicated to Boussinesq.  Here, we show that, beside the
validity of the modeling assumption~\eqref{eq:reynolds_stress}, the Reynolds stress tensor
$\reyn$ is dissipative, under minimal assumptions on the regularity of the data of the
problem.
\subsection{Time-averaging of uniformly-local  fields} 
We list in this section some technical properties of the operator $M_t$ acting on
uniform-local fields, and the corresponding global weak solutions to the NSE.  The first
result is the following
\begin{lemma} 
  \label{lem:lemme_Mt_ulc} 
  Let $1 < p < \infty$ and let be given $f \in L^{p}_{uloc}(\R_+;X)$. Then
  \begin{equation*} 
    \forall \, t \ge 1, \qquad  \| M_t (f) \|_X \le 2 \|f\|_{L^{p}_{uloc}(\R_+;X)}. 
  \end{equation*} 
\end{lemma}
\begin{proof} 
  Applying~\eqref{eq:1} and some straightforward inequalities yields
  \begin{equation*}
    \|M_{t}(f)\|_{X} 
    \leq\frac{1}{t}\int_{0}^{t}\| f \|_{X}\,ds
    \leq
    \frac{1}{t}\int_{0}^{[t]+1}\|f\|_{X}\,ds 
    \leq\frac{1}{t}\sum_{k=0}^{[t]}\int_{k}^{k+1}\| f \|_{X}\,ds.
  \end{equation*}
  Therefore by the H\"older inequality we get:
  \begin{equation*}
    \begin{aligned}
      \|M_{t}(f)\|_{X} &   \zoom
      \leq\frac{1}{t}\sum_{k=0}^{[t]}\left(\int_{k}^{k+1}\| f \|_{X}^{p}\,ds\right)^{1/p} 
      \left(\int_{k}^{k+1}1\,ds\right)^{1/p'} 
      \\
      & \zoom \leq\frac{[t]+1}{t}\| f \|_{L^{p}_{uloc}(X)}^p \leq2 \| f
      \|_{L^{p}_{uloc}(\R_+;X)}^p, 
    \end{aligned}
\end{equation*}
the last inequality being satisfied since $\frac{[x]+1}{x}\leq 2$ is valid for all
$x\geq1$.
\end{proof}

In the next section, we will focus on the case $p=2$ and $X = V'$. We will need the
following result, which is a consequence of Lemma~\ref{lem:energy-estimates}.
\begin{lemma}
  \label{lem:lemme_dissispation} 
  Let be given $\bv_{0}\in H$, ${\bf f} \in L^2_{uloc}(\R_+;V')$, and let $\vv$ be a
  global weak solution to the NSE corresponding to the above data. Then, we have, $\forall\,t\geq1$, 
 \begin{eqnarray} 
   \label{eq:dissipation_mean} 
   && M_t(\|\nabla\bv\|^2) \leq
    \frac{\|\bv_{0}\|^{2} }{\nu
      t}+2\frac{\mathcal{F}^{2}}{\nu^2}, \\
    &&    \label{eq:dissipation_mean_2}  M_t( \| \fv \|_{V'} ^2 ) \le 2 \mathcal{F}^{2}, 
  \end{eqnarray} 
  where ${\cal F} = \| {\bf f} \|_{L^2_{uloc}(\R_+;V') }$.
\end{lemma}
\begin{proof}
  It suffices to divide estimate~\eqref{eq:dissipation} by $\nu t$.  Therefore, it follows that
\begin{equation*}
  \frac{1}{t} \int_0^t\|\nabla \bv(s)\|^2\,ds\leq
  \frac{\|\bv_0\|^2}{\nu t}+2\frac{\mathcal{F}^2}{\nu^2}.
\end{equation*}
Estimate~\eqref{eq:dissipation_mean_2}  is straightforward. 
\end{proof}
In particular the family $\{ M_t (\vv) \}_{t \in \R_+}$ is bounded in $V$. Therefore, we
can as of now state the following, which will be recalled in a more precise form in the
next sections.
\begin{corollary} 
  There exists $\vm$ such that -up to a sub-sequence- $M_t (\vv) \to \vm$ as $t\to\infty$,
  in appropriate topologies.
\end{corollary} 
The following result is a direct consequence of~\eqref{eq:dissipation_mean}
and~\eqref{eq:dissipation_mean_2} combined with Cauchy-Schwarz inequality.
\begin{corollary} 
  The family $\{ M_t( < \fv , \vv > ) \}_{t >0}$ is bounded uniformly in $t$, and one has
  \begin{equation}
    \label{eq:M_t_bounded}  
    \big| M_t( < \fv , \vv > ) \big| \le \sqrt 2 \mathcal{F} \left ( \frac{\|\bv_{0}\|^{2} }{\nu
        t}+2\frac{\mathcal{F}^{2}}{\nu^2} \right )^{\frac{1}{2}}.
  \end{equation}
\end{corollary} 
We finish this section with a last technical result, that we will need to prove 
Item \ref{it:foias_generalized})  of Theorem~\ref{thm:main_theorem}.
\begin{lemma}
  \label{lem:fbar=ftilde} 
  Let $1 < p < \infty$ and let be given $f \in L^{p}_{uloc}(\R_+;X)$, which satisfies in
  addition
  \begin{equation} 
    \label{eq:hyp_Mt} 
    \exists \   \widetilde {f}\in X, \quad \text{such that} \quad
    \lim_{t\to+\infty}\int_{t}^{t+1}\| f(s)-\widetilde f\|_{X}^p\,ds=0. 
  \end{equation} 
  Then, we have
  \begin{equation}
    \label{eq:limite_ftilde}
    \lim_{t\to+\infty} \frac{1}{t}\int_{0}^{t }\|f (s) - \widetilde{ f }\|_{X}^{p}\,ds=0.
\end{equation} 
Moreover, $M_t(f)$ weakly converges to $\widetilde f$ in $X$ when $t \to+\infty$. In
particular, we have $\widetilde f = \overline f$.
\end{lemma}
\begin{proof} 
  By the hypothesis~(\ref{eq:hyp_Mt}), we have that 
\begin{equation*}
  \forall\,\varepsilon>0\quad \exists\, M\in \N:\quad \int_{t}^{t+1} \|f (s) -
  \widetilde f\|_{X}^{p}\,ds<\frac{\varepsilon}{2}\qquad \forall\,t>M.
\end{equation*}
Hence, for  $t \geq M$, then 
\begin{equation*}
  \begin{aligned}
    \frac{1}{t}\int_{0}^{t}\|f(s) -
      \widetilde{f}\|_{X}^{p}\,ds=\frac{1}{t}\int_{0}^{M}\|f (s) -
      \widetilde{f}\|_{X}^{p}\,ds+\frac{1}{t}\int_{M}^{t}\|f(s) -
      \widetilde{f}\|_{X}^{p}\,ds
    \\ 
    \leq\frac{M}{t}\big ( \| f \|_{L^{p}_{uloc}(X)}^p+\|\widetilde{ f }\|_{X}^{p}\big )
    +\frac{[t ]+1-M}{t}\frac{\varepsilon}{2}.
  \end{aligned}
\end{equation*}
It follows that one can choose $M$ large enough such that
\begin{equation*}
  \frac{1}{t}\int_{0}^{t}
  \|f (s)-\widetilde { f }\|_{X}^{p}\,ds<\varepsilon\qquad \forall\, t>M,
\end{equation*}
hence, being this valid for arbitrary $\varepsilon>0$, it
follows~(\ref{eq:limite_ftilde}). 

It remains to prove the weak convergence of $M_t(f)$ to $\widetilde f\in X$ when
$t\to+\infty$. To this end, let be given $\varphi \in X'$. Then, we have
\begin{equation*}
  < \varphi, M_t(f) > - < \varphi, \widetilde f >\ = \frac{1}{t} \int_0^t
  <\varphi, f(s)- \widetilde f > \,ds, 
\end{equation*}
which leads to 
\begin{equation*}
  \big| < \varphi, M_t(f) > - < \varphi, \widetilde f > \big| \le  \frac{1}{t} 
  \int_0^t  \| \varphi \|_{X'} \|  f(s)- \widetilde f \|_X \,ds, 
\end{equation*}
and by H\"older inequality, 
\begin{equation*}
  \left| < \varphi, M_t(f) > - < \varphi, \widetilde f > \right| 
  \le \| \varphi\|_{X'} \left(  \frac{1}{t}  \int_0^t  \|  f(s)- \widetilde f \|_X^p \,ds \right )^{\frac{1}{p}}, 
\end{equation*}
yielding, by~(\ref{eq:limite_ftilde}), to $\zoom \lim_{t \to+\infty } < \varphi, M_t(f)>\ =\
<\varphi, \widetilde f >$, hence concluding the proof.
\end{proof} 
The following corollary definitively concludes Item~\ref{it:foias_generalized}) of
Theorem~\ref{thm:main_theorem}.
\begin{corollary} 
  Let be given $\bv_{0}\in H$, ${\bf f} \in L^2_{uloc}(\R_+;V')$ that
  satisfies~\eqref{eq:hyp_Mt}, and let $\vv$ be a global weak solution to the NSE
  corresponding to the above data. Moreover let $\vm$ be such that $ \zoom \lim_{t \to
    \infty} M_t \vv = \vm$ in $V$ (eventually up to a sub-sequence), then
\begin{equation*}
\lim_{t \to \infty} M_t ( < \fv , \vv > ) = \, < \overline \fv, \vm >.
\end{equation*}
\end{corollary} 
\begin{proof} 
 Let us write the following decomposition:
  \begin{equation*}
    \frac{1}{t}\int_{0}^{t}<\bff,\bv>\,ds=
    \frac{1}{t}\int_{0}^{t}<\bff-
    \overline{\bff},\bv>\,ds+\frac{1}{t}\int_{0}^{t}< \overline{\bff},\bv>\,ds. 
  \end{equation*}
  On one hand since $\mean{\bff}\in V$ is independent of $t$, we obviously have
  \begin{equation*} 
    \frac{1}{t}\int_{0}^{t}<\overline{\bff},\bv>\,ds\to\  <\mean{\bff},\mean{\bv}>.
  \end{equation*} 
  On the other hand, we have also 
  \begin{equation}
    \label{eq:ine_final}
    \begin{aligned}
      \left \vert \frac{1}{t}\int_{0}^{t}<\bff- \overline{\bff},\bv>\,ds
      \right \vert& \leq \frac{1}{t}\int_{0}^{t}\|\bff- \overline{\bff}\|_{V'}\|\nabla
      \bv\|\,ds
      \\
      &\hspace{-2cm}\leq\left(\frac{1}{t}\int_{0}^{t}\|\bff-
        \overline{\bff}\|_{V'}^{2}\,ds\right)^{1/2}\left(\frac{1}{t}\int_{0}^{t}\|\nabla
        \bv\|^{2}\,ds\right)^{1/2}.
    \end{aligned}
  \end{equation} 
  Combining~\eqref{eq:limite_ftilde} with~\eqref{eq:dissipation_mean} shows that the
  right-hand side in~\eqref{eq:ine_final} vanishes as $t \to \infty$. \end{proof} 
\begin{remark}
  It is important to observe that 
  \begin{equation*}
    M_t\big (<\bff,\bv>\big )\to \ <\overline{\bff},\overline{\bv}>,
  \end{equation*}
  is --in some sense-- an assumption on the (long-time) behavior of the ``covariance''
  between the external force and the solution itself.  Cf.~Layton~\cite{Lay2014} for a
  related result in the case of ensemble averages.

  The control of the (average/expectation of) kinetic energy in terms of the energy input
  is one of the remarkable features of classes of statistical solutions, making the
  stochastic Navier-Stokes equations very appealing in this context. See the review, with
  applications to the determination of the Lilly constant, in Ref.~\cite{BF2010}. See
  also~\cite{Fla2005}.
\end{remark}

\section{Proof of Theorem~\ref{thm:main_theorem}}
\label{sec:proof-main1}
In all this section we have as before $\vv_0 \in H$, ${\bf f} \in L^2_{uloc}(\R_+;V')$,
and $\vv$ is a global weak solution to the NSE~\eqref{eq:NSE} corresponding to the above
data.  We split the proof of Theorem~\ref{thm:main_theorem} into two steps. We first apply
the operator $M_t$ to the NSE, then we extract sub-sequences and take the limit in the
equations. In the second step we make the identification with the Reynolds stress $\reyn$
and show that it is dissipative in average, at least when ${\bf f}$ satisfies in
addition~\eqref{eq:additional_ass_f}.
\subsection{Extracting sub-sequences} 
We set:
\begin{equation*}
 \bV_t(\bx):=M_t(\bv)(\bx).
\end{equation*}
Applying the operator $M_t$ on the NSE we see that for almost all $t\geq0$ and for all
$\bphi\in V$, the field $\bV_t$ is a weak solution of the following steady Stokes problem
(where $t>0$ is simply a parameter)
\begin{equation}
  \label{eq:Reynolds_weak_form}
  \begin{aligned}
    \nu\int_\Omega\nabla \bV_t:\nabla\bphi\,d\bx+\int_{\Omega}M_t((\bv\cdot\nabla)\,
    \bv)\cdot\bphi\,d\bx&=\ < M_t(\bff),\bphi>
    \\
   &+ \int_\Omega\frac{\bv_0-\bv(t)}{t}\cdot\bphi\,d\bx.
  \end{aligned}
\end{equation}
The full justification of the equality~\eqref{eq:Reynolds_weak_form} starting from the
definition of global weak solutions can be obtained by following a very well-known path
used for instance to show with a lemma by Hopf that Leray-Hopf weak solutions can be
re-defined on a set of zero Lebesgue measure in $[0,t]$ in such a way that $\bv(s)\in H$
for \textit{all} $s\in[0,t]$, see for instance Galdi~\cite[Lemma~2.1]{Gal2000a}. In fact,
by following ideas developed among the others by Prodi~\cite{Pro1959}, one can take
$\chi_{[a,b]}$ the characteristic function of an interval $[a,b]\subset\R$, and use as
test function its regularization multiplied by $\bphi\in V$. Passing to the limit as the
regularization parameter vanishes one gets~\eqref{eq:Reynolds_weak_form}.

The process of extracting sub-sequences, which is the core of the main result, is reported
in the following proposition.
\begin{proposition}
  \label{prop:thm1}
  Let be given a global solution $\bv$ to the NSE, corresponding to the data $\bv_{0}\in
  H$ and $\bff\in L^{2}_{uloc}(\R_{+};V')$. Then, there exist
  \begin{enumerate}[a)]
  \item a sequence $\{t_n\}_{n\in \N}$ that goes to $+\infty$
    when $n$ goes to $+\infty$;
  \item a vector field $\mean{\bff} \in V'$; 
 \item a vector field $\mean{\bv} \in V$;
 \item a vector field $\bB\in L^{3/2}(\Omega)^3$;  

such that such that it holds when $n \to \infty$:
   \begin{equation*}
     \begin{aligned}
       &M_{t_{n}}(\bff)\rightharpoonup \mean{\bff} \qquad \text{in } V',
       \\
       &M_{t_{n}}(\bv)\rightharpoonup \mean{\bv} \qquad \text{in } V,
       \\
       &M_{t_{n}}\big ((\bv\cdot\nabla)\,\bv\big )\rightharpoonup\bB
       \qquad \text{in } L^{3/2}(\Omega)^3\subset V',
     \end{aligned}
   \end{equation*}
and for all $\bphi\in V$
\begin{equation}
  \label{eq:Reynolds0}
  \nu\int_\Omega\nabla\mean{\bv}:\nabla\bphi\,d\bx
  +\int_\Omega\bB\cdot\bphi\,d\bx=\ <\mean{\bff},\bphi>.   
\end{equation}
Moreover, by defining
\begin{equation} 
  \label{eq:def_of_F} 
  \bF:=\bB-(\mean{\bv}\cdot\nabla)\,\mean{\bv}\in L^{3/2}(\Omega)^3,
\end{equation} 
we can also rewrite~\eqref{eq:Reynolds0} as follows
\begin{equation}
  \label{eq:Reynolds}
  \nu\int_\Omega\nabla \mean{\bv}:\nabla\bphi\,d\bx+
  \int_\Omega(\mean{\bv}\cdot\nabla)\,\mean{\bv}\cdot\bphi\,d\bx
  +\int_\Omega\bF\cdot\bphi\,d\bx=\ < \mean{\bff},\bphi>;
\end{equation}
\item writing $\vv' = \vv- \vm$, $\fv' = \fv - \overline \fv$, we also have
\begin{equation*}
M_{t_n} ( < \fv, \vv > ) \to \,  <\overline \fv, \vm > + \overline{ < \fv', \vv' > }.
\end{equation*}
     \end{enumerate} 
  \end{proposition}
\begin{proof}[Proof of Proposition~\ref{prop:thm1}]
  As ${\bf f} \in L^2(\R_+;V')$, we deduce from Lemma~\ref{lem:lemme_Mt_ulc} that
  $\{M_t({\bf f})\}_{t>0}$ is bounded in $V'$.  Hence, we can use weak pre-compactness of
  bounded sets in the Hilbert space $V'$ to infer the existence of $t_n$ and
  $\mean{\bff}\in V'$ such that $M_{t_n}(\bff)\rightharpoonup \mean{\bff}$ in $V'$.
  Next, estimate~\eqref{eq:dissipation_mean} from Lemma~\ref{lem:lemme_dissispation},
  combined with estimate~\eqref{eq:1} from Lemma~\ref{lem:bochner}, leads to the bound
  \begin{equation*}
    \exists\,c>0:\qquad  \|\nabla M_{t}(\bv)\| = \|M_{t}(\nabla\bv)\|\leq c \qquad\forall\,t>0,
  \end{equation*}
  proving (up to a the extraction of a further sub-sequence from $\{t_n\}$, which we call
  with the same name) that $M_{t_{n}}(\bv)\rightharpoonup\mean{\bv}$ in $V'$.

  Then, we observe that, if $\bv\in L^{\infty}(0,T;H)\cap L^{2}(0,T;V)$ by classical
  interpolation
\begin{equation*}
  (  \bv\cdot\nabla)\,\bv\in L^{r}(0,T;L^{s}(\Omega))\qquad \text{with
  }\quad\frac{2}{r}+\frac{3}{s}=4,\quad  r\in[1,2].
\end{equation*}
In particular, we get
\begin{equation*}
  \|  (
  \bv\cdot\nabla)\,\bv\|_{L^{3/2}(\Omega)}\leq\|\bv\|_{L^{6}}\|\nabla
  \bv\|_{L^{2}}\leq C_{S}\|\nabla\bv\|^{2},
\end{equation*}
where $C_{S}$ is the constant of the Sobolev embedding $H^{1}_{0}(\Omega)\hookrightarrow
L^{6}(\Omega)$. Hence, by using the bounds~\eqref{eq:energy_1}-\eqref{eq:dissipation} on
the weak solution $\bv$ we obtain that
\begin{equation*}
\exists\,c:\qquad  \|M_{t}\big ((\bv\cdot\nabla)\,\bv)\big )\|_{L^{3/2}(\Omega)}\leq
  c,\qquad\forall\,t>0,
\end{equation*}
proving that, up to a further sub-sequence relabelled again as $\{t_n\}$,
\begin{equation*}
  M_{t_{n}}\big ((\bv\cdot\nabla)\,\bv\big )\rightharpoonup \bB \qquad
  \text{in } L^{3/2}(\Omega)^{3}, 
\end{equation*}
for some vector field $\bB\in L^{3/2}(\Omega)^{3}$.

Next, we use~\eqref{eq:energy_1} which shows that
\begin{equation*}
      \int_\Omega\frac{\bv_0-\bv(t)}{t}\cdot\bphi\,d\bx\to0\qquad
      \text{as }t\to+\infty.
\end{equation*}
Then, writing the weak formulation and by using the results of weak convergence previously
proved, we get~\eqref{eq:Reynolds0}. Then, the identity~\eqref{eq:Reynolds} comes simply
from the definition~(\ref{eq:def_of_F}) of $\bF$.

It remains to prove the last item. We know from~\eqref{eq:M_t_bounded} that the sequence $
\{ M_{t_n} (< \fv, \vv > ) \} _{n \in \N}$ is bounded in $\R$. By extracting again a
sub-sequence (still denoted by $\suite t n$), we can get a convergent sequence still
denoted (after relabelling) by $ \{ M_{t_n} (< \fv, \vv > ) \} _{n \in \N}$, and let
$\overline { < \fv, \vv > }$ be its limit. Let us write the decomposition
\begin{equation}
  \label{eq:decomposition_fluxes} 
  \begin{array}{l} M_{t_n} (< \fv, \vv >) 
    \\ 
    \\
    \hskip 1cm 
    =<\overline \fv, \vm > + M_{t_n} (< \fv', \vm >) + M_{t_n} (< \overline \fv, \vv' >)+
    M_{t_n} ( < \fv', \vv' > ). 
  \end{array} 
\end{equation}
As $M_{t_n} (< \fv', \vm >) = \, < M_{t_n} ( \fv'), \vm >$, we deduce from the results
above that $M_{t_n} (< \fv', \vm >) \to 0$ as $n \to \infty$. Similarly, we also have
$M_{t_n} (< \overline \fv, \vv' >) \to 0$. Hence, we deduce
from~\eqref{eq:decomposition_fluxes} that $\{ M_{t_n} ( < \fv', \vv' > ) \}_{n \in \N}$ is
convergent, and if we denote by $\overline {< \fv', \vv' >}$ its limit, the following
natural decomposition holds true:
\begin{equation}
  \label{eq:reynolds_stress_ext}
  \overline { < \fv, \vv >} = <\overline \fv, \vm >+\overline {< \fv', \vv' >}, 
\end{equation}
concluding the proof.
\end{proof}
\subsection{Reynolds stress, energy balance and dissipation}
\label{sec:Reynolds-stress}
In the first step we have identified a limit $(\mean{\bv},\mean{\bff})$ for the
time-averages of both velocity and external force ($\bv,\bff)$. We need now to recast this
in the setting of the Reynolds equations, in order to address the proof of the Boussinesq
assumption.
\begin{proof}[Proof of Theorem~\ref{thm:main_theorem}]
  Beside the results in Proposition~\ref{prop:thm1}, in order to complete the proof of
  Theorem~\ref{thm:main_theorem}, we have to prove the following facts:
  \begin{enumerate}[1)]
  \item \label{it:item1} the proper identification of the limits with the Reynolds stress
    $\reyn$;
  \item\label{it:item2} the energy balance for $\mean{\bv}$;
  \item \label{it:item3} to prove that~\eqref{eq:reyn_dissipative2}  holds, namely
  \begin{equation*}
    \E =  \nu  \mean{\|\nabla{\bv}'\|^2}\le \int_\Omega(\nabla\cdot
    \reyn)\cdot\mean{\bv}\,d\bx + \overline {< \fv', \vv' > }.
  \end{equation*}  
\end{enumerate}
We proceed in the same order.

{\sl Item \ref{it:item1}}.  Since $\bv\in L^2(0,T;V)\subset L^2(0,T;L^6(\Omega)^3)$, it
  follows that $\bv\otimes\bv\in L^1(0,T;L^3(\Omega))$. Hence, the same argument as in the
  previous subsection shows that (possibly up to the extraction of a further sub-sequence)
  there exists a second order tensor $\boldsymbol{\theta} \in L^{3}(\Omega)^{9}$ such that
  \begin{equation*}
    M_{t_n}(\bv\otimes\bv)\rightharpoonup \boldsymbol{\theta} \qquad\text{in }L^{3}(\Omega)^{9}.
  \end{equation*}
  Let us set
  \begin{equation*}
    \reyn:=\boldsymbol{\theta}-\mean{\bv}\otimes\mean{\bv}.
  \end{equation*}
  Since the operator $M_t$ commutes with the divergence operator, the
  equation~\eqref{eq:Reynolds_weak_form} becomes
  \begin{equation}
    \label{eq:Reynolds_weak_form2}
    \begin{aligned}
      \nu\int_\Omega\nabla \bV_t:\nabla\bphi\,d\bx- \int_{\Omega}M_t(\vv \otimes \vv): \g
      \bphi\,d\bx&=\ < M_t(\bff),\bphi>
      \\
      &+ \int_\Omega\frac{\bv_0-\bv(t)}{t}\cdot\bphi\,d\bx.
    \end{aligned}
  \end{equation}
  Then, by taking the limit along the sequence $t_{n}\to+\infty$
  in~(\ref{eq:Reynolds_weak_form2}), we get\footnote{
    {According to the formal decomposition~(\ref{eq:decomposition_Reynolds_stress}), this
      suggests that $M_{t_n}(\vv' \otimes \vv') \to 0$, provided that one is able to give
      a rigorous sense and sufficiently strong bounds on $\vv'\otimes \vv'$, for the weak
      solution $\vv$.
}} the equality
  \begin{equation*}
    \bF=\nabla\cdot\reyn.
  \end{equation*}
  \medskip

  {\sl Item \ref{it:item2}}.  We use $\mean{\bv}\in V$ in~\eqref{eq:R} as test function
  and we obtain the equality
  \begin{equation} 
    \label{eq:energy_balance_moy} 
    \nu\|\nabla\mean\bv\|^2+\int_\Omega
    (\nabla\cdot\reyn)\cdot\mean{\bv}\,d\bx=\ <\mean{\bff},\mean{\bv}>.
  \end{equation}  
  We observe that due to the absence
  of the time-variable the following identity concerning the integral over the space
  variables is valid
  \begin{equation*}
    \int_\Omega (\mean{\bv}\cdot\nabla)\,\mean{\bv}\cdot\mean{\bv}\,d\bx=
    \int_\Omega (\mean{\bv}\cdot\nabla)\,\frac{|\mean{\bv}|^2}{2}\,d\bx=0\qquad
    \forall\,\mean{\bv}\in V.
  \end{equation*}
  This is one of the main technical facts which are typical of the mathematical analysis
  of the steady Navier-Stokes equations and which allow to give precise results for the
  averaged Reynolds equations. On the other hand, we recall that if $\bv(t,\bx)$ is a
  non-steady (Leray-Hopf) weak solution, then the space-time integral 
  \begin{equation*}
    \int_0^T\int_\Omega
    ({\bv}\cdot\nabla)\,\bv\cdot{\bv}\,d\bx\,dt,
  \end{equation*}
  is not well defined and consequently the above integral vanishes only formally.
\medskip 

  {\sl Item~\ref{it:item3}}. From now, we assume that the
  assumption~\eqref{eq:additional_ass_f} in the statement of
  Theorem~\ref{thm:main_theorem} holds true. We integrate the energy
  inequality~\eqref{eq:energy} between $0$ and $t_{n}$ and we divide the result by $t_n>0$,
  which leads to
  \begin{equation}
    \label{eq:energy_averaged2}
    \frac{\|\bv(t)\|^{2}}{2t_n}+\frac{1}{t_n}\int_{0}^{t_n}\|\nabla
    \bv(s)\|^{2}\,ds\leq   
    \frac{\|\bv_{0}\|^{2}}{2t_n}+\frac{1}{t_n}\int_{0}^{t_n}<\bff,\bv>\,ds.
  \end{equation}
  Recall that by Lemma~\ref{lem:energy-estimates}
  \begin{equation*}
    \frac{\|\bv(t)\|^{2}}{2t}\to0\quad\text{and}\quad
    \frac{\|\bv_{0}\|^{2}}{2t}\to0\quad\text{as }t\to+\infty.
  \end{equation*}
   Therefore, we take the limit in~\eqref{eq:energy_averaged2} and we use~\eqref{eq:reynolds_stress_ext}, which yields
 \begin{equation*}
\nu \overline { \| \g \vv \|^2} \le \overline {< \fv , \vv >} =  \, <\overline \fv, \vm > + \, \overline {< \fv', \vv' >}.
\end{equation*}
By~\eqref{eq:energy_balance_moy}  we then have
\begin{equation*}
\nu \overline { \| \g \vv \|^2} \le  \nu\|\nabla\mean\bv\|^2+\int_\Omega
    (\nabla\cdot\reyn)\cdot\mean{\bv}\,d\bx +  \, \overline {< \fv', \vv' >},
  \end{equation*}
  which yields~\eqref{eq:reyn_dissipative2}  by~\eqref{eq:new_Foias}, concluding the
  proof.
 \end{proof}
\section{On ensemble averages}
\label{sec:ensemble}
In this section we show how to use the results of Theorem~\ref{thm:main_theorem} to give
new insight to the analysis of ensemble averages of solutions. In this case we study
suitable averages of the long-time behavior and not the long-time behavior of statistics,
as in Layton~\cite{Lay2014}.

Since we first take long-time limits and then we average the Reynolds
equations, the initial datum is not so relevant. In fact due to the fact that it holds
\begin{equation*}
  \frac{  \|\bv_0\|^2}{t}\to0\qquad \text{as }t\to+\infty,
\end{equation*}
then the mean $\mean{\bv}$ is not affected by the initial datum.

As claimed in the introduction, we consider now the problem of having several external
forces, say a whole family $\{\bff^k\}_{k\in\N}\subset V'$, all independent of time. We
can think as different experiments with slightly different external forces, whose
difference can be due to errors in measurement or in the uncertainty intrinsic in any
measurement method. In particular, one can consider for a given force $\bff$ and that
$\{\bff^k\}$ will represent small oscillations around it, hence we can freely assume that
we have an uniform bound
\begin{equation}
  \label{eq:bound_f_k}
  \exists\,C>0:\qquad \|\bff^k\|_{V'}\leq C\qquad \forall\,k\in\N.
\end{equation} 
Having in mind this physical setting, we denote by $\mean{\bv^k}\in V$ the long-time
average of the solution corresponding to the external force $\bff^k\in V'$ and, as
explained before (without loss of generality) to the initial datum $\bv_0=\mathbf{0}$. The
vector $\mean{\bv^k}\in V$ satisfies for all $\bphi\in V$ the following equivalent
equalities for all $k\in\N$
\begin{align*}
  & \nu\int_\Omega\nabla \mean{\bv^k}:\nabla\bphi\,d\bx+
  \int_\Omega\bB^k\cdot\bphi\,d\bx=\ < {\bff^k},\bphi>,
  \\
  &\nu\int_\Omega\nabla \mean{\bv^k}:\nabla\bphi\,d\bx+
  \int_\Omega(\mean{\bv^k}\cdot\nabla)\,\mean{\bv^k}\cdot\bphi\,d\bx
  +\int_\Omega\bF^k\cdot\bphi\,d\bx=\ < {\bff^k},\bphi>,
\end{align*}
for appropriate $\bB^k,\bF^k\in L^3(\Omega)^{3/2}$. Since both $V$ and $V'$ are Hilbert
spaces, by using~\eqref{eq:bound_f_k} it follows that there exists $\smean{\bff}\in V'$
and a sub-sequence (still denoted by $\{\bff^k\}$) such that
\begin{equation*}
  \bff^k\rightharpoonup \smean{\bff}\qquad \text{ in }V'.
\end{equation*}
Our intention is to characterize, if possible, the limit of
$\{\mean{\bv}^k\}_{k\in\N}$. If the forces are fluctuations around a mean value, then the
field $\mean{\bv^k}$ will remain bounded in $V$, but possibly without converging to some
limit. From an heuristic point of view one can expect that averaging the sequence of
velocities (which corresponds to averaging the result over different realizations) one can
identify a proper limit, which retains the ``average'' effect of the flow. 

Again, it comes into the system the main idea at the basis of Large Scale methods: The
average behavior of solutions seems the only quantity which can be measured or
simulated. 

It is well-known that one of the most used \textit{summability technique} is that of
Ces\`aro and consists in taking the mean values, hence we focus on the arithmetic mean of
time-averaged velocities
\begin{equation*}
  \mathbf{S}^n:=\frac{1}{n}\sum_{k=1}^n \mean{\bv^k}.
\end{equation*}
It is a basic calculus result that if a real sequence $\{x_j\}_{j\in\N}$ converges to
$x\in\R$, then also its Ces\`aro mean $S_n = \frac{1}{n}\sum_{j=1}^n x_j$ will converge to
the same value $x$. On the other hand, the converse is false; sufficient conditions on the
sequence $\{x_j\}_{j\in\N}$ implying that if the Ces\`aro mean converges, then the
original sequence converges, are known in literature as Tauberian theorems. This is a
classical topic in the study of divergent sequences/series.  In the case of $X$-valued
sequences $\{\bu^k\}_{k\in\N}$ (the space $X$ being an infinite dimensional Banach space)
one has again that if a sequence converges strongly or weakly, then its Ces\`aro mean will
converge to the same value, strongly or weakly in $X$, respectively.

The fact that averaging generally improves the properties of a sequence, is reflected also
in the setting of Banach spaces even if with additional features coming into the theory.
Two main results we will consider are two theorems known as Banach-Saks and Banach-Mazur.

Banach and Saks originally in 1930 formulated the result in
$L^p(0,1)$, but it is valid in more general Banach spaces.
\begin{theorem}[Banach-Saks]
  Let be given a bounded sequence $\{x_j\}_{j\in\N}$ in a reflexive Banach space
  $X$. Then, there exists a sub-sequence $\{x_{j_k}\}_{k\in\N}$ such that the sequence
  $\{S_m\}_{m\in \N}$ defined by
  \begin{equation*}
    S_m := \frac{1}{m}\sum_{k=1}^m x_{j_k},
  \end{equation*}
converges strongly in $X$.
\end{theorem}
The reader can observe that in some cases it is not needed to extract a
sub-sequence (think of any orthonormal set in an Hilbert space, which
is weakly converging to zero, and the Ces\`aro averages converge to
zero strongly), but in general one cannot infer that the averages of
the full sequence converge strongly. One sufficient condition is that
of \textit{uniform weak convergence}. We recall that $\{x_j\}\subset X$
\textit{uniformly weakly} converges to zero if for any $\epsilon>0$
there exists $j\in\N$, such that for all $\phi\in X'$, with
$\|\phi\|_{X'}\leq1$, it holds true that
\begin{equation*}
  \#\big\{ j\in N:\  |\phi(x_j)|\geq\varepsilon \big\}\leq j.
\end{equation*}
See also Brezis~\cite[p.~168]{Bre2011}.
\medskip

Another way of improving the weak convergence to the strong one is by the by the
convex-combination theorem (cf.~Yosida~\cite[p.120]{Yos1980}).
\begin{theorem}[Banach-Mazur]
  Let $(X,\|\,.\,\|_X)$ be a Banach space and let $\{x_j\}\subset X$ be a sequence
  such that $x_j\rightharpoonup x$ as $j\to+\infty$.

  Then, one can find for each $n\in\N$, real coefficients $\{\alpha^n_j\}$, for
  $j=1,\dots,n$ such that
  \begin{equation*}
    \alpha^n_j\geq0\qquad \text{and }\qquad\sum_{j=1}^n\alpha^n_j=1,
  \end{equation*}
  such that 
  \begin{equation*}
    \sum_{j=1}^n\alpha^n_j x_j\rightarrow x\qquad\text{in }X,\quad \text{as}\quad
    n\to+\infty, 
  \end{equation*}
  that is we can find a ``convex combination'' of $\{x_j\}$, which strongly converges to
  $x\in X$.
\end{theorem}
One basic point will be that of considering averages of the external forces, which we will
denote by $\smean{\bff}^n$ and considering the same averages of the solution of the
Reynolds equations $\smean{\mean{\bv}}^n$.  They are both bounded and hence, weakly
converging (up to a sub-sequence) to $\smean{\bff}\in V'$ and $\smean{\mean{\bv}}\in V$,
respectively. Then, in order to prove that the dissipativity is preserved one has to
handle the following limit of the products
\begin{equation*}
  \lim_{n\to+\infty}  <\smean{\bff}^n,\smean{\mean{\bv}}^n>,
\end{equation*}
which cannot be characterized, unless (at least) one of the two terms
converges strongly. This is why we have to use special means instead
of the simple Ces\`aro averages

The first result of this section is then the following:
\begin{proposition}
  \label{prop:ensemble}
  Let be given $\{\bff^k\}_{k\in\N}$ uniformly bounded in $V'$. Then one can find either a
  Banach-Saks sub-sequence or a convex combination of $\{\bv^k\}_{k\in \N}$, which are
  converging weakly to some $\smean{\bv}\in V$, which satisfies a Reynolds
  system~\eqref{eq:Reynolds6}, with an additional dissipative term.
\end{proposition}
\begin{proof}[Proof of Theorem~\ref{prop:ensemble}]
  We define $\smean{\bff}^n$ and $\smean{\bv}^n$ to be either 
  \begin{equation*}
    \smean{\bff}^n:=\frac{1}{n}\sum_{k=1}^n{\bff^{j_k}} \qquad
    \text{and}\qquad      \smean{\bv}^n:=\frac{1}{n}\sum_{k=1}^n{\mean{\bv}^{j_k}},
  \end{equation*}
  or alternatively
  \begin{equation*}
    \smean{\bff}^n:=\sum_{j=1}^n\alpha^n_j{\bff^j} \qquad
    \text{and}\qquad      \smean{\bv}^n:=\sum_{j=1}^n\alpha^n_j\mean{\bv^j},
  \end{equation*}
  where the sub-sequence $\{{j_k}\}_{k\in\N}$ or the coefficients
  $\{\alpha^n_j\}_{j,\,n\in \N}$ are chosen accordingly to the Banach-Saks of Banach-Mazur
  theorems in such a way that in both cases
\begin{equation*}
   \smean{\bff}^n\to\smean{\bff}\qquad \text{in }V'.
\end{equation*}

We define, accordingly to the same rules $\smean{\bB}^n$, and we observe
that, by linearity, we have $\forall\,n\in\N$
  \begin{equation}
    \label{eq:Reynolds3}
    \nu\int_\Omega\nabla \smean{\bv}^n:\nabla\bphi\,d\bx+
    \int_\Omega\smean{\bB}^n\cdot\bphi\,d\bx=\ < 
    \smean{\bff}^n,\bphi>\qquad\forall\,\bphi\in V.
  \end{equation}
  Then, we can define
  $\smean{\bF}^n:=\smean{\bB}^n-(\smean{\bv}^n\cdot\nabla)\,\smean{\bv}^n$, to
  rewrite~\eqref{eq:Reynolds3} also as follows
  \begin{equation}
    \label{eq:Reynolds4}
    \nu\int_\Omega\nabla \smean{\bv}^n:\nabla\bphi\,d\bx+
    \int_\Omega\big(\smean{\bv}^n\cdot\nabla\big)\,\smean{\bv}^n\cdot\bphi\,d\bx+\int_\Omega\smean{\bF}^n\cdot\bphi\,d\bx
    =\  <\smean{\bff}^n,\bphi>.
  \end{equation}
  By the uniform bound on $\|\bff^k\|_{V'}$ and by results of
  Section~\ref{sec:Reynolds-stress} on the Reynolds equations it
  follows that there exists $C$ such that $\|\mean{\bv}^k\|_V\leq C$,
  hence 
  \begin{equation*}
    \| \smean{\bv}^n\|_V  \leq C\qquad\forall\,n\in\N,
  \end{equation*}
  and  we can suppose that (up to sub-sequences) we have weak
  convergence of the convex combinations
  \begin{equation*}
    \begin{aligned}
      \smean{\bv}^n&\rightharpoonup
      \smean{\bv}\qquad\text{ in }V,
      \\
      \smean{\bB}^n&\rightharpoonup
      \smean{\bB}\qquad\text{ in }L^{3/2}(\Omega)^3,
      \\
      \smean{\bF}^n&\rightharpoonup
      \smean{\bF}\qquad\text{ in }L^{3/2}(\Omega)^3,
    \end{aligned}
  \end{equation*}
  Hence, passing to the limit in~\eqref{eq:Reynolds3}, we obtain
  \begin{equation*}
    \nu\int_\Omega\nabla\smean{\bv}:\nabla\bphi\,d\bx+
    \int_\Omega\smean{\bB}\cdot\bphi\,d\bx=\ < 
    \smean{\bff},\bphi>\qquad\forall\,\bphi\in V.
  \end{equation*}
  By the same reasoning used before we have, for
  $\smean{\bF}:=\smean{\bB}-(\smean{\bv}\cdot\nabla)\,\smean{\bv}$,
  \begin{equation}
    \label{eq:Reynolds6}
    \nu\int_\Omega\nabla\smean{\bv}:\nabla\bphi\,d\bx+
    \int_\Omega(\smean{\bv}\cdot\nabla)\,\smean{\bv}\cdot\bphi\,d\bx+
    \int_\Omega\smean{\bF}\cdot\bphi\,d\bx=\ <  
    \smean{\bff},\bphi>.
  \end{equation}
  Then, if we take $\bphi=\smean{\bv}$ in~\eqref{eq:Reynolds6} we obtain
  \begin{equation}
    \label{eq:energy-statistical}
    \nu\|\nabla\smean{\bv}\|^2
    +\int_\Omega\smean{\bF}\cdot\smean{\bv}\,d\bx=\ <\smean{\bff},\smean{\bv}>. 
  \end{equation}
  On the other hand, if we take $\bphi=\smean{\bv}^n$ in~\eqref{eq:Reynolds4} and by the
  result of the previous section, we have
  \begin{equation*}
    \nu\|\nabla\smean{\bv}^n\|^2\leq\ <\smean{\bff}^n,\smean{\bv}^n>,
  \end{equation*}
  hence passing to the limit, by using the strong convergence of
  $\smean{\bff}^n$ in $V'$ and the weak convergence of $\smean{\bv}^n$
  in $V$ we have
  \begin{equation*}
    \nu\|\nabla\smean{\bv}\|^2\leq\liminf_{n\to+\infty}
    \nu\|\nabla\smean{\bv}^n\|^2\leq\ 
    <\smean{\bff},\smean{\bv}>.
  \end{equation*}
  If we compare with~\eqref{eq:energy-statistical} we have finally the dissipativity
  \begin{equation*}
    \frac{1}{|\Omega|}\int_\Omega(\nabla\cdot \smean{\reyn})\cdot\smean{\bv}\,d\bx
    =\frac{1}{|\Omega|}\int_\Omega\smean{\bF}\cdot\smean{\bv}\,d\bx\geq0, 
  \end{equation*}
  that is a sort of ensemble/long-time Boussinesq
  hypothesis, cf. with the results from Ref.~\cite{Lay2014,LL2016}.
\end{proof}
In the previous theorem, we have a result which does not concern directly with the
ensemble averages, but a selection of special coefficients is required. This is not
completely satisfactory from the point of view of the numerical computations, where the
full arithmetic mean should be considered. The main result can be obtained at the price of
a slight refinement on the hypotheses on the external forces

To this end we recall a lemma, which is a sort of Rellich theorem in
negative spaces (see also Galdi~\cite[Thm.~II.5.3]{Gal2011} and
Feireisl~\cite[Thm.~2.8]{Fei2004}).
\begin{lemma}
\label{lem:Sobolev}
Let $\Omega\subset\R^n$ be bounded and let be given $1<p<n$. Let $\{f_k\}_{k\in\N}$ be a
sequence uniformly bounded in $L^q(\Omega)$ with $q>(p^*)'$, where $p^*=\frac{n p}{n-p}$ is
the exponent in the Sobolev embedding $W^{1,p}_0(\Omega)\hookrightarrow
L^{p^*}(\Omega)$. Then, there exists a sub-sequence $\{f_{k_m}\}_{m\in \N}$ and $f\in
L^q(\Omega)$ such that
  \begin{equation*}
    \begin{aligned}
     & f_{k_m}\rightharpoonup f\qquad\text{in }L^q(\Omega),
     \\
     & f_{k_m}\rightarrow f\qquad\text{in }W^{-1,p'}(\Omega),
    \end{aligned}
  \end{equation*}
  or, in other words, the embedding $L^q(\Omega)\hookrightarrow\hookrightarrow
  W^{-1,p'}(\Omega)$ is compact.
\end{lemma}
We present the proof for the reader's convenience.
\begin{proof}[Proof of Lemma~\ref{lem:Sobolev}] 
Since by hypothesis $L^q(\Omega)$ is reflexive, by the
Banach-Alaouglu-Bourbaki theorem we can find a sub-sequence  $f_{k_m}$
such that 
\begin{equation*}
  f_{k_m}\rightharpoonup f\qquad\text{in }L^q(\Omega),
\end{equation*}
and by considering the sequence $\{ f_{k_m}- f\}_{m\in\N}$ we can suppose that $f=0$.  We
then observe that by the Sobolev embedding we have the continuous embeddings
\begin{equation*}
  W^{1,p}_0(\Omega)\hookrightarrow  L^{p^*}(\Omega)\sim
  L^{(p^*)'}(\Omega)\hookrightarrow  (W^{1,p}_0(\Omega))'\simeq
  W^{-1,p'}(\Omega),
\end{equation*}
where $L^{p^*}(\Omega)\sim L^{(p^*)'}(\Omega)$ is the duality identification, while the
second one $(W^{1,p}_0(\Omega))'\simeq W^{-1,p'}(\Omega) $ is the Lax isomorphism. This
shows $L^{(p^*)'}(\Omega)\subseteq W^{-1,p'}(\Omega)$.

Next, let be given a sequence $\{f_{k_m}\}\subset W^{-1,p'}(\Omega)$, then
by reflexivity (since $1<p<\infty$) there exists $\{\phi_{k_m}\}\subset
W^{1,p}_0(\Omega)$ such that
\begin{equation*}
    \|f_{k_m}\|_{W^{-1,p'}(\Omega)}=f_{k_m}(\phi_{k_m})=\ <
    f_{k_m},\phi_{k_m}>,
  \end{equation*} {with}
  $\|\phi_{k_m}\|_{W^{1,p}_0(\Omega)}=\|\nabla\phi_{k_m}\|_{L^{p}(\Omega)}=1$.

Hence, by using the classical Rellich theorem, we can find a
sub-sequence $\{\phi_{k_j}\}_{j\in\N}$ such that
\begin{equation*}
  \phi_{k_j}\to \phi \qquad\text{in }L^r(\Omega)\qquad\forall\,r<p^*.
\end{equation*}
In particular, we fix $r=q'$ (observe that $q>(p^*)'$ implies
$q'<p^*$) and we have 
\begin{equation*}
  \|f_{k_m}\|_{W^{-1,p'}(\Omega)}=\ < f_{k_m},\phi_{k_m}-\phi>+< f_{k_m},\phi>.
\end{equation*}
The last term converges to zero, by the definition of weak convergence
$f_{k_m}\rightharpoonup0$, while the first one satisfies
\begin{equation*}
  \left|<    f_{k_m},\phi_{k_m}-\phi>\right|\leq\|f_{k_m}\|_{W^{-1,p'}}\|\phi_{k_m}-\phi\|_{W^{1,p}_0},
\end{equation*}
and since $\|f_{k_m}\|_{W^{-1,p'}}$ is uniformly bounded and
$\|\phi_{k_m}-\phi\|_{W^{1,p}_0}$ goes to zero, then also this one vanishes as
$j\to+\infty$.
\end{proof}
\begin{proof}[Proof of Theorem~\ref{thm:main_ensemble}]
  The proof of this theorem can be obtained by following the same ideas of the
  Proposition~\ref{prop:ensemble}.  In fact, the main improvement is that the weak
  convergence $\bff^j\rightharpoonup \smean{\bff}$ in $L^q(\Omega)$ implies (without
  extracting sub-sequences) that
  \begin{equation*}
    \bff^k\rightarrow \bff\qquad  \text{in }V'.
  \end{equation*}
  This follows since from any sub-sequence we can find a further
sub-sequence which is converging strongly, by
Lemma~\ref{lem:Sobolev}. Then, by the weak convergence of the original
sequence, the limit is always the same and this implies that the whole
sequence $\{\bff^k\}$ strongly converges to its weak limit.

Hence, we have
\begin{equation*}
  \frac{1}{n}\sum_{k=1}^n   \bff^k\rightarrow \bff\qquad\text{in }V',
\end{equation*}
and then, since $\smean{\bv}^n\rightharpoonup\smean{\bv}$ in $V$, we can infer that
\begin{equation*}
  < \smean{\bff}^n,\smean{\bv}^n >\ =\  <\frac{1}{n}\sum_{k=1}^n   \bff^k,
  \frac{1}{n}\sum_{k=1}^n\mean{\bv}^k>\ \to\   < \smean{\bff},\smean{\bv} >,
\end{equation*}
and the rest follows as in Proposition~\ref{prop:ensemble}.
\end{proof}

\bigskip

\section*{Acknowledgments}
The research that led to the present paper was partially supported by a grant of the group
GNAMPA of INdAM the project the University of Pisa within the grant
PRA$\_{}2018\_{}52$~UNIPI \textit{Energy and regularity: New techniques for classical PDE
problems.}
\def\ocirc#1{\ifmmode\setbox0=\hbox{$#1$}\dimen0=\ht0 \advance\dimen0
  by1pt\rlap{\hbox to\wd0{\hss\raise\dimen0
  \hbox{\hskip.2em$\scriptscriptstyle\circ$}\hss}}#1\else {\accent"17 #1}\fi}
  \def\cprime{$'$} \def\polhk#1{\setbox0=\hbox{#1}{\ooalign{\hidewidth
  \lower1.5ex\hbox{`}\hidewidth\crcr\unhbox0}}} \def\cprime{$'$}

\end{document}